\documentclass[10pt]{amsart}
\usepackage{texmaX,semmaX,semtkzX}
\usepackage{makecell}
\usepackage{arydshln} 

\def\mod{\quad\textup{mod }}
\def\tmod{\ \textup{mod }}
\advance\oddsidemargin -1cm
\advance\evensidemargin -1cm
\advance\textwidth 2cm
\advance\headsep -0.3cm
\advance\textheight 0.3cm 

\usepackage{tabularx}

\usepackage{blkarray}

\begin{document}

\title{Selmer ranks under quadratic twists satisfying the Heegner hypothesis}

\author{Alexandros Konstantinou}
  \address{Alexandros Konstantinou, Institute of Mathematics and Informatics, Bulgarian Academy of Sciences, Acad. G. Bonchev St., bl.8, 1113 Sofia, Bulgaria}
  \email{a.konstantinou@math.bas.bg}


\marginpar{}

\begin{abstract}
We investigate variations of Selmer ranks under quadratic twists satisfying the Heegner hypothesis. In particular, starting with an elliptic curve $E/\mathbb{Q}$ with partial $2$-torsion and a common relaxed Selmer group, we derive explicit formulae describing the effect of twisting on Selmer ranks in terms of matrices over $\mathbb{F}_{2}$. As an application, we show that these formulae are compatible with predictions made by the parity conjecture.
\end{abstract}

\maketitle

\setcounter{tocdepth}{1}

\tableofcontents
\maketitle

\section{Introduction}

\subsection{Elliptic curves over quadratic extensions} \label{subsec_intro1}
A classical problem in the arithmetic of elliptic curves is to understand how the Mordell--Weil rank changes over quadratic extensions. To study this, consider an elliptic curve $E/\mathbb{Q}$ and a quadratic field $K=\mathbb{Q}(\sqrt{d})$. The action of the nontrivial element $\sigma \in \Gal(K/\mathbb{Q})$ naturally splits the group $E(K)$ into $(\pm 1)$-eigenspaces, which correspond to $E(\mathbb{Q})$ and the quadratic twist $E^{d}(\mathbb{Q})$. This decomposition yields the formula:
\[
\rk E(K)=\rk E(\mathbb{Q})+\rk E^{d}(\mathbb{Q}).
\]
Consequently, the problem of studying ranks over quadratic extensions is entirely encoded in the ranks of the quadratic twists $E^{d}/\mathbb{Q}$.

The study of quadratic twists lies at the heart of modern arithmetic statistics. Prominent among the conjectures in this area is that of Goldfeld, which predicts that within the quadratic twist family of an elliptic curve over~$\mathbb{Q}$, half of the twists have analytic rank~$0$ and half have analytic rank~$1$.
In a recent trilogy of papers~\cite{SmithI, SmithII, SmithBSDGoldfeld}, it has been established that the Birch and Swinnerton-Dyer conjecture implies Goldfeld’s conjecture, thereby linking two of the deepest conjectures in the arithmetic of elliptic curves.

All investigations of this kind typically proceed via the more accessible Selmer groups, since the computation of the Mordell--Weil rank is obscured by the mysterious nature of the Tate--Shafarevich group. 
Early work includes~\cite{HB94}, which shows that for the congruent number curve $y^{2}=x^{3}-x$, the $2$-Selmer ranks of its quadratic twists follow a remarkably regular distribution. 
Subsequent results~\cite{Kane10, SD08} extend this to elliptic curves with full rational $2$-torsion and no cyclic $4$-isogeny over~$\mathbb{Q}$, while~\cite{KMR13} establishes an analogous theorem over arbitrary number fields in the generic $2$-torsion case. 
In sharp contrast, when $E(\mathbb{Q})[2]\cong\mathbb{Z}/2\mathbb{Z}$, corresponding to the intermediate case of partial $2$-torsion, no distribution function exists for the $2$-Selmer ranks within quadratic twist families~\cite{KlagsbrunLemkeOliver}. More recently, similar questions have been investigated for higher-dimensional abelian varieties ~\cite{Morgan2019, Yu2016, Yu2019}.

Beyond statistical questions, quadratic extensions satisfying the \emph{Heegner hypothesis}, meaning imaginary quadratic fields $K/\Q$ in which every prime of bad reduction for $E$ splits, offer a particularly rich setting for studying rational points. 
For such $K$, the theory of Heegner points provides a systematic supply of rational points on $E$, arising from the modular parametrisation, and among these one distinguishes a canonical point $P_{K}\in E(K)$. The work of Gross–Zagier and Kolyvagin show that if $P_{K}$ is non-torsion, then $\operatorname{rk}(E/K)=1$ and $\Sha(E/K)$ is finite.
To descend these results to $\mathbb{Q}$, one appeals to analytic work on $L$-functions in quadratic twist families~\cite{BFH90, MM91, Wa85}, yielding the strongest known case of the Birch and Swinnerton–Dyer conjecture: if the analytic rank of $E/\mathbb{Q}$ satisfies $\mathrm{rk}_{\mathrm{an}}(E/\mathbb{Q})\le1$, then  
\[
\mathrm{rk}_{\mathrm{an}}(E/\mathbb{Q})=\operatorname{rk}E(\mathbb{Q})
\quad\text{and}\quad
\Sha(E/\mathbb{Q})\ \text{is finite.}
\]

This leads us to a natural and delicate intersection: considering quadratic extensions that satisfy the Heegner hypothesis for elliptic curves with partial $2$-torsion. In this specific scenario, the behaviour of the full $2$-Selmer group is governed by the interaction between the $\varphi$-- and $\hat{\varphi}$--Selmer groups.

\subsection{Effect of twists satisfying the Heegner hypothesis}In what follows, we write $E/\mathbb{Q}$ for an elliptic curve with$E(\mathbb{Q})[2]\cong \mathbb{Z}/2\mathbb{Z}$. We may and do assume that after a change of variables $E$ is given by an integral Weierstrass model
$y^{2}=x(x^{2}+a x+b)$ with $a,b\in\mathbb{Z}$.
The point $(0,0)$ generates $E(\mathbb{Q})[2]$ and defines a $2$–isogeny $\varphi:E\to E’$, with dual isogeny $\hat{\varphi}:E’\to E$.

Our main result focuses on twists by integers $d$ arising from $\mathcal{D}_{E/\Q}$ (see Definition~\ref{def:primes} below). These all define imaginary quadratic fields satisfying the Heegner hypothesis for $E/\mathbb{Q}$. In particular, we give an explicit formula describing how the $\varphi$– and $\hat{\varphi}$–Selmer ranks change under such twists in terms of matrices over $\mathbb{F}_{2}$ built from the Legendre symbols of the primes dividing $d$.

The variation in Selmer ranks depends only on the sign of the discriminant $\Delta_E$, the coefficients $a,b$ defining $E$, and on whether the ``relaxed'' Selmer group $\Sel_{\mathcal{S}_d}^{\varphi}(E/\Q)$ (see~\S\ref{subsec_truncated}) contains a suitable integer -- a condition that is straightforward to verify in practice (see Remark~\ref{remark:-1_in_truncated_Selmer}).
\begin{theorem}[\textup{=\,Thm. ~\ref{thm:master2}, Cor.~\ref{cor:hatphi-variation-explicit}}]\label{thm:intro-variation}
Let $E/\Q$ be given by the equation $y^{2}=x(x^{2}+a x+b)$, and let $d\in\mathcal{D}_{E/\Q}$.
Define
\[
\delta_{\varphi}(d)=\mathrm{rk}_{\varphi}(E^{-d}/\Q)-\mathrm{rk}_{\varphi}(E/\Q),
\qquad
\delta_{\hat{\varphi}}(d)=\mathrm{rk}_{\hat{\varphi}}(E^{\prime -d}/\Q)-\mathrm{rk}_{\hat{\varphi}}(E’/\Q).
\]
The values of $\delta_{\varphi}(d)$ and $\delta_{\hat{\varphi}}(d)$ are given in the following table.
\medskip
\renewcommand{\arraystretch}{1.2}
\[
\begin{array}{c|c|l}
\hline
\delta_{\varphi}(d) & \delta_{\hat{\varphi}}(d) & \text{Conditions} \\ 
\hline
0 & \omega(d) & b<0 \\[3pt]
\omega(d) & 0 & \Delta_E<0,\ b>0 \\[4pt]
\omega(d)-1-\operatorname{rank}_{\F_2}(\mathbf{A}_{d})
& \omega(d)-\operatorname{rank}_{\F_2}(\mathbf{A}_{d})
& \Delta_E>0,\ b>0,\ a<0,\ \eta_{\varphi}(d)=1 \\[4pt]
\omega(d)-\operatorname{rank}_{\F_2}(\widehat{\mathbf{A}_{d}})
& \omega(d)+1-\operatorname{rank}_{\F_2}(\widehat{\mathbf{A}_{d}})
& \Delta_E>0,\ b>0,\ a<0,\ \eta_{\varphi}(d)=0 \\[4pt]
\omega(d)+1-\operatorname{rank}_{\F_2}(\widetilde{\mathbf{A}_{d}})
& \omega(d)-\operatorname{rank}_{\F_2}(\widetilde{\mathbf{A}_{d}})
& \Delta_E>0,\ b>0,\ a>0,\ \eta_{\varphi}(d)=1 \\[4pt]
\omega(d)-\operatorname{rank}_{\F_2}(\overline{\mathbf{A}_{d}})
& \omega(d)-1-\operatorname{rank}_{\F_2}(\overline{\mathbf{A}_{d}})
& \Delta_E>0,\ b>0,\ a>0,\ \eta_{\varphi}(d)=0 \\[2pt]
\hline
\end{array}
\]

\medskip
Here, $\omega(d)$ denotes the number of prime factors of $d$, and all matrices $\mathbf{A}_{d}$, $\widetilde{\mathbf{A}_{d}}$, $\widehat{\mathbf{A}_{d}}$ and $\overline{\mathbf{A}_{d}}$ are defined over $\F_2$ in Notation~\ref{not:twist-matrices}.
\end{theorem}

\subsection{Structure of the argument} \label{subsec_intro_4}
Given a $2$–isogeny $\varphi:E\to E'$, the Selmer group $\Sel^{\varphi}(E/\Q)$
is defined as the subgroup of $H^{1}(\Q,E[\varphi])$ defined by the certain local conditions at each place of~$\Q$.
In particular, for each place $v$, let
\[
\delta_{\varphi,v}\colon E'(\Q_{v})/\varphi(E(\mathbb{Q}_{v})) \to H^{1}(\Q_{v},E[\varphi])
\]
be the local Kummer map, and write
$
\mathcal{L}_{\varphi,v}(E)\ :=\ \operatorname{im}\bigl(\delta_{\varphi,v}\bigr)
\ \subseteq\ H^{1}(\Q_{v},E[\varphi]).
$
Then the $\varphi$–Selmer group is the subgroup of $H^{1}(\mathbb{Q},E[\varphi]$ cut out by these local subgroups:
\[
\Sel^{\varphi}(E/\Q)
=\left\{\, c\in H^{1}(\Q,E[\varphi]) :
\operatorname{res}_{v}(c)\in\mathcal{L}_{\varphi,v}(E)\text{ for all }v\,\right\}.
\]
Thus, the effect of twisting on $\Sel^{\varphi}(E/\Q)$ is determined entirely by how the local subgroups $\mathcal{L}_{\varphi,v}(E)$ change under quadratic twisting.  
To study this, we consider certain \emph{relaxed} versions of Selmer groups, obtained by omitting the conditions at the primes dividing $d$ and at infinity:
\[
\Sel_{\mathcal{S}_d}^{\varphi}(E/\Q)
=\{\, c\in H^{1}(\Q,E[\varphi]) :
\operatorname{res}_{v}(c)\in\mathcal{L}_{\varphi,v}(E)\text{ for all }v\nmid d\infty\,\},
\]
and similarly for $E^{-d}$. Under the natural isomorphism $E[\varphi]\cong E^{-d}[\varphi]$, the groups
$\Sel^{\varphi}(E/\Q)$, $\Sel^{\varphi}(E^{-d}/\Q)$, and their relaxed variants all embed in the same cohomology space $H^{1}(\Q,E[\varphi])$. As subgroups of this cohomology group, one has
\[
\Sel_{\mathcal{S}_d}^{\varphi}(E/\Q)
=\Sel_{\mathcal{S}_d}^{\varphi}(E^{-d}/\Q),
\]
which is entirely natural: since $E/\Q_{p}$ and $E^{-d}/\Q_{p}$ are isomorphic for all $p\mid 2N_{E}$, and have good reduction at every prime not dividing $d$ and the real place, their local conditions coincide at all places defining the relaxed Selmer groups. Consequently, any discrepancy between $\Sel^{\varphi}(E/\Q)$ and $\Sel^{\varphi}(E^{-d}/\Q)$ can only arise from the primes dividing $d$ and infinity.

To this end, we take the common relaxed Selmer group
$\Sel_{\mathcal{S}_{d}}^{\varphi}(E/\Q)=\Sel_{\mathcal{S}_{d}}^{\varphi}(E^{-d}/\Q)$
as a common starting point, and successively impose the remaining local conditions at the places \hbox{$v\mid d\infty$}.
By tracking how these local conditions change under twisting (see Lemmas~\ref{loc_conditions_at_q} and~\ref{lem:infinity}),
we deduce the effect that the quadratic twist has on the Selmer group.
To achieve this, we introduce the homomorphisms
\[
\varepsilon_{\infty},\ \lambda_{p},\ \gamma_{p}:\Q^{\times}/\Q^{\times2}\to\F_{2},
\]
which capture the effect of imposing the local conditions at the remaining primes. 

To make this concrete, consider the case $b<0$.  
Per Lemmas~\ref{loc_conditions_at_q} and~\ref{lem:infinity}, the effect of twisting on the local conditions is as follows: for each prime $q\mid d$, the local subgroup at $q$ for $E^{-d}$ is trivial, whereas for $E$ it is the unramified subgroup $H^{1}_{\textup{ur}}(\mathbb{Q}_{q},E[\varphi])$; at the real place, both local subgroups are trivial. 
Hence, starting from the common relaxed Selmer group $\Sel^{\varphi}_{\mathcal{S}_{d}}(E/\Q)$, 
the usual Selmer groups $\Sel^{\varphi}(E/\Q)$ and $\Sel^{\varphi}(E^{-d}/\Q)$ 
are recovered by performing the following steps at the places $v\mid d\infty$:
\[
\renewcommand{\arraystretch}{1.2}
\begin{array}{|c|c|l|}
\hline
\text{Curve} & \text{Place} & \text{Operation} \\ \hline
E & q\mid d & \text{intersect with } \ker(\lambda_{q}) \\ 
  & \infty  & \text{intersect with } \ker(\varepsilon_{\infty}) \\ \hline
E^{-d} & q\mid d & \text{intersect with } \ker(\lambda_{q})\cap\ker(\gamma_{q}) \\ 
       & \infty  & \text{intersect with } \ker(\varepsilon_{\infty}) \\ \hline
\end{array}
\]
This procedure yields an explicit description of $\Sel^{\varphi}(E/\Q)$ and $\Sel^{\varphi}(E^{-d}/\Q)$ 
in terms of the common relaxed group $\Sel^{\varphi}_{\mathcal{S}_{d}}(E/\Q)$ 
(see Theorem~\ref{thm:master}), 
thereby reducing the study of Selmer–rank variation to a problem in linear algebra over~$\F_{2}$. We note that this method of expressing Selmer conditions in terms of matrices over $\mathbb{F}_{2}$ and analysing their variation under twisting fits into a broader framework developed in earlier work. We note that the use of matrices over $\mathbb{F}_{2}$ to describe Selmer conditions and their variation under twisting has appeared in earlier works, including~\cite{HB94} and~\cite{KK17}, which treat, respectively, the congruent number curve and elliptic curves over~$\mathbb{Q}$ with rational partial $2$–torsion.

\subsection{Parity conjecture and twisting}
\label{subsec_parity_conjecture_intro}

The Birch and Swinnerton--Dyer conjecture predicts that the algebraic and analytic ranks of an elliptic curve $E$ over a number field $F$ coincide.  
Writing $\Lambda(E,s)$ for the completed $L$--function of $E/F$, it is expected that $\Lambda(E,s)$ extends analytically to $\mathbb{C}$ and satisfies the functional equation  
$
\Lambda(E,2-s)=\omega(E/F)\Lambda(E,s),
$
where $\omega(E/F)=\prod_v \omega_v(E/F)$ is the product of local root numbers.  
Assuming analytic continuation, the Birch and Swinnerton--Dyer conjecture then implies the \emph{parity conjecture}.

\begin{conjecture}[Parity conjecture]\label{conj:parity}
For an elliptic curve $E$ over a number field $F$,
\[
(-1)^{\mathrm{rk}(E/F)} = w(E/F).
\]
\end{conjecture}

Progress towards this conjecture has come through its $p$-primary refinement, the \emph{$p$-parity conjecture}, obtained by replacing the algebraic rank with the $\mathbb{Z}_{p}$-corank of the $p^{\infty}$-Selmer group.
This version is known to hold for all primes $p$ and all elliptic curves over~$\mathbb{Q}$ \cite[Thm.~1.4]{DD10}.

Quadratic twisting provides a natural setting in which to study parity phenomena.  
A central question in this direction was formulated by Kramer and Tunnell~\cite[Conj.~3.1]{KT82}, 
who conjectured a relation between local root numbers and local arithmetic data, 
specifically Tamagawa numbers.  
This conjecture was later resolved for elliptic curves in~\cite[Cor.~4.8]{DD09}, 
thereby establishing the $2$-parity conjecture for elliptic curves over quadratic extensions: 
for any quadratic extension $L/F$ and any elliptic curve $E/F$,  
\begin{equation} \label{2-parity_over_quadratic}
(-1)^{\mathrm{rk}_{2^{\infty}}(E/L)} = w(E/L).
\end{equation}
An analogue of this result for Jacobians of hyperelliptic curves, under mild assumptions, 
was subsequently obtained in~\cite{Morgan2023}.
In particular, if $L=\mathbb{Q}(\sqrt{-d})$ satisfies the Heegner hypothesis for $E/\mathbb{Q}$, 
then a root number computation gives $w(E^{-d}/\mathbb{Q})=-w(E/\mathbb{Q})$.  
In light of~\eqref{2-parity_over_quadratic}, this implies that
\[
\mathrm{rk}_{2^{\infty}}(E^{-d}/\mathbb{Q}) - \mathrm{rk}_{2^{\infty}}(E/\mathbb{Q}) \equiv 1 \tmod{2},
\]
and consequently
\begin{equation}\label{eq:parity-expected}
\mathrm{rk}_{2}(E^{-d}/\mathbb{Q}) - \mathrm{rk}_{2}(E/\mathbb{Q}) \equiv 1 \tmod{2},
\end{equation}
(see Theorem~\ref{thm:parity}). When $E(\mathbb{Q})=\Z/2\Z$, variations in $2$-Selmer ranks under twisting are controlled by the $\varphi$- and $\hat{\varphi}$-Selmer ranks.
For each square-free integer $d$, set
\[
\delta_{\varphi}(d)
 = \mathrm{rk}_{\varphi}(E^{-d}/\mathbb{Q})
   - \mathrm{rk}_{\varphi}(E/\mathbb{Q}), \qquad
\delta_{\hat{\varphi}}(d)
 = \mathrm{rk}_{\hat{\varphi}}(E^{\prime -d}/\mathbb{Q})
   - \mathrm{rk}_{\hat{\varphi}}(E'/\mathbb{Q}).
\]
It is then natural to ask whether the explicit formulae for $\delta_{\varphi}(d), \delta_{\hat{\varphi}}(d)$ detailed in Theorem~\ref{thm:intro-variation} are consistent with the parity relation in~\eqref{eq:parity-expected}.  
The answer is affirmative: we show that for any $d\in\mathbb{Q}^{\times}/\mathbb{Q}^{\times2}$ (not necessarily $d\in\mathcal{D}_{E/\mathbb{Q}}$), the combined variation $\delta_{\varphi}(d)+\delta_{\hat{\varphi}}(d)$ captures precisely the parity shift in $2$--Selmer ranks (see Proposition~\ref{prop:delta_phi}):
\begin{equation} \label{eq:selmer-parity}
\mathrm{rk}_{2}(E^{-d}/\mathbb{Q})
 - \mathrm{rk}_{2}(E/\mathbb{Q})
 \equiv
 \delta_{\varphi}(d) + \delta_{\hat{\varphi}}(d)
 \mod{2}.
\end{equation}
Moreover, when $d\in \mathcal{D}_{E/\mathbb{Q}}$ the explicit expressions in Theorem~\ref{thm:intro-variation} satisfy
\[
\delta_{\varphi}(d) + \delta_{\hat{\varphi}}(d) \equiv 1 \tmod{2},
\]
independently of the matrix ranks. As a result, we recover exactly the parity shift predicted by the parity conjecture for quadratic twists satisfying the Heegner hypothesis. At the level of Tate--Shafarevich groups, the compatibility detailed in \eqref{eq:selmer-parity} gives rise to a symmetric relation between the $\varphi$- and $\hat{\varphi}$-kernels.

\begin{corollary}[{=~Corollary~\ref{cor:symmetric-sha-ratio}}]
Let $d$ be a square-free integer, and assume that $\Sha(E/\mathbb{Q}(\sqrt{-d}))[2^{\infty}]$ is finite. Then
\[
\frac{\#\Sha(E/\mathbb{Q})[\varphi]}{\#\Sha(E'/\mathbb{Q})[\hat{\varphi}]}
\equiv
\frac{\#\Sha(E^{-d}/\mathbb{Q})[\varphi]}{\#\Sha(E^{\prime -d}/\mathbb{Q})[\hat{\varphi}]}
\mod{\mathbb{Q}^{\times 2}}.
\]
If this congruence fails for some~$d$, then at least one of $\Sha(E/\mathbb{Q})[2^{\infty}]$ or $\Sha(E^{-d}/\mathbb{Q})[2^{\infty}]$ must be infinite.
\end{corollary}

\subsection*{Acknowledgements}
The author would like to thank Evis Ieronymou for valuable discussions which led to the work presented here, and Adam Morgan for helpful comments on an earlier draft and for drawing attention to related works in the literature.
This work was supported by the scientific programme 
``Enhancing the Research Capacity in Mathematical Sciences (PIKOM)'',
No.~DO1-67/05.05.2022, 
of the Ministry of Education and Science of Bulgaria.

\subsection{General Notation } \label{General_notation}
Throughout, we fix an elliptic curve $E/\mathbb{Q}$ with partial $2$–torsion, that is, $E(\mathbb{Q})[2]\cong\mathbb{Z}/2\mathbb{Z}$.  
We shall always regard $E/\mathbb{Q}$ as given by an integral Weierstrass equation
\[
E:\ y^{2}=x(x^{2}+a x+b), \qquad A,B\in\mathbb{Z},
\]
so that $(0,0)$ is the nontrivial rational $2$–torsion point.  
This determines a $2$–isogeny $\varphi:E\to E'$, where $E'$ is the $2$–isogenous curve.  We write $\hat{\varphi}: E' \to E$ to denote the dual isogeny.

\medskip
\noindent
For ease of reference, we collect the main notation used throughout the paper.
\medskip

\medskip

\renewcommand{\arraystretch}{1.35}
\begin{tabular}{|p{0.23\textwidth}|p{0.65\textwidth}|} 
\hline
$E$ & elliptic curve with $E(\mathbb{Q})[2]=\mathbb{Z}/2\mathbb{Z}$\\
$E',\,E^{d},\,E'^{d}$ & respectively the $2$–isogenous curve, the quadratic twist by $d$, and the quadratic twist of $E'$ by $d$\\
$N_E,\,\Delta_E,\,(\Delta_E)_{\textup{min}}$ & conductor of $E$, discriminant of the chosen model of $E/\mathbb{Q}$, and minimal discriminant of $E/\mathbb{Q}$\\
$\varphi,\,\hat{\varphi}$ & $2$–isogeny with kernel $\langle(0,0)\rangle$ and its dual $E'\to E$\\[2pt]
$\Sel^{\varphi}(E/\Q)$ & $\varphi$–Selmer group\\
$\Sel_{S}^{\varphi}(E/\Q)$ & truncated $\varphi$–Selmer group (Def.~\ref{def:truncated_selmer})\\
$\mathrm{rk}_{\varphi}(E/\Q),\,\mathrm{rk}_{2}(E/\Q)$ & $\F_{2}$–dimensions of $\Sel^{\varphi}(E/\Q)$ and $\Sel^{2}(E/\Q)$\\
$\mathcal{L}_{\varphi,v}(E),\,\mathcal{L}_{\hat{\varphi},v}(E)$ & local subgroups at $v$ (see~\eqref{eq:local_conditioss})\\[2pt]
$\mathcal{Q}_{E/\Q},\,\mathcal{D}_{E/\Q}$ & twisting sets (Def.~\ref{def:primes})\\
$\mathcal{S}_{d},\,\mathcal{S}_{d}^{0}$ & sets of primes dividing $d$ (Not.~\ref{not:primesofd})\\
$\mathcal{T}_{d}^{\pm}$ & distinguished subsets of $\mathcal{S}_{d}$ (Not.~\ref{not:T_d})\\
$(\cdot,\cdot)_{v}$ & quadratic Hilbert symbol over $\Q_{v}$\\[2pt]
$\epsilon_{\infty},\,\gamma_{q},\,\lambda_{q}$ & maps $\Q^{\times}/\Q^{\times2}\to\F_{2}$ encoding signs, valuations, and Legendre symbols (Def.~\ref{def:homomorphisms})\\
$\Gamma_{\mathcal{T}},\,\Lambda_{\mathcal{T}},\,S_{\infty}$ & matrices representing $\gamma_{q},\,\lambda_{q},\,\epsilon_{\infty}$ (Def.~\ref{def:matrices})\\
$\mathbf{A}_{d},\,\widetilde{\mathbf{A}}_{d},\,\overline{\mathbf{A}}_{d},\,\widehat{\mathbf{A}}_{d}$ & block matrices defined via Legendre symbols (Not.~\ref{not:twist-matrices})\\[2pt]
$\Sigma$ & supporting set for $V\subseteq\Q^{\times}/\Q^{\times2}$ (Not.~\ref{not:GammaLambda})\\
$\omega(d)$ & number of distinct prime factors of $d$\\
\hline
\end{tabular}
\medskip

\section{Selmer groups and local conditions}
We briefly recall the construction of the $\varphi$-Selmer group.

\subsection{Selmer groups and local subgroups} 
 There is a $2$–isogeny $\varphi: E \to E'$, given by $(x,y)\mapsto (y^{2}/x^{2},\, y(b-x^{2})/x^{2})$ with kernel $E[\varphi]= \langle (0,0) \rangle$, where $E': y^{2} = x^{3} - 2a x^{2} + (a^{2}-4b)x$. Then, $\varphi(E[2])= \langle (0,0) \rangle$, where $(0,0)$ is now a point in $E'(\mathbb{Q})$.  

This yields the short exact sequence $0 \to E[\varphi] \to E[2] \xrightarrow{\varphi} \varphi(E[2]) \to 0$, and hence a connecting map $\delta_{\varphi}: E'(\mathbb{Q})/\varphi(E(\mathbb{Q})) \to H^{1}(\mathbb{Q},C)$.  
Identifying $E[\varphi] \cong \mu_{2}$ via $(0,0) \mapsto -1$ gives an isomorphism $H^{1}(\mathbb{Q},E[\varphi])\cong \mathbb{Q}^{\times}/\mathbb{Q}^{\times 2}$ and a homomorphism $\mu_{\varphi}: E'(\mathbb{Q})/\varphi(E(\mathbb{Q})) \to \mathbb{Q}^{\times}/\mathbb{Q}^{\times 2}$ determined by
\begin{equation} \label{eq:mu_phi}
\mu_{\varphi}(P)=
\begin{cases}
x, & \text{if } P=(x,y) \text{ with } x\neq 0, \\[6pt]
\Delta_{E}, & \text{if } P=(0,0),
\end{cases}
\end{equation}
where $\Delta_{E}$ is the discriminant of $E/\mathbb{Q}$ (see \cite[Prop.~X.4.9]{SilvermanAEC}). As discussed in \S \ref{subsec_intro_4}, this construction carries over locally. For each place $v$ of $\mathbb{Q}$ there is a connecting homomorphism
\begin{equation} \label{eq:local_conditioss}
\delta_{\varphi,v}:
E'(\mathbb{Q}_{v})/\varphi(E(\mathbb{Q}_{v}))
\to
H^{1}(\mathbb{Q}_{v},E[\varphi]),
\qquad
\mathcal{L}_{\varphi,v}(E/\mathbb{Q})
:=\operatorname{im}(\delta_{\varphi,v}),
\end{equation}
whose image defines the distinguished local condition
$\mathcal{L}_{\varphi,v}(E/\mathbb{Q})\subseteq H^{1}(\mathbb{Q}_{v},E[\varphi])$. Via the canonical identification 
$H^{1}(\mathbb{Q}_{v},E[\varphi])\cong\mathbb{Q}_{v}^{\times}/\mathbb{Q}_{v}^{\times2}$,
these can be viewed as subgroups of $\mathbb{Q}_{v}^{\times}/\mathbb{Q}_{v}^{\times2}$.  
The $\varphi$–Selmer group is then given by
\[
\Sel^{\varphi}(E/\mathbb{Q})
=\bigcap_{v}\ker\!\left(
H^{1}(\mathbb{Q},E[\varphi])
\to
H^{1}(\mathbb{Q}_{v},E[\varphi])/\mathcal{L}_{\varphi,v}(E/\mathbb{Q})
\right).
\]
As a result, determining the $\varphi$–Selmer group amounts to understanding the local subgroups $\mathcal{L}_{\varphi,v}(E)$ for all places $v$ simultaneously.  
For odd places $v$ of good reduction, these are described by the following classical result.
\begin{lemma}[{\cite[Lemma~4.1]{Cassels1965}}] \label{lem:unramified_condition}
Suppose that $E/\mathbb{Q}$ has good reduction at an odd prime $p$.  
Under the identifications 
$H^{1}(\mathbb{Q},E[\varphi]) \cong \mathbb{Q}^{\times}/\mathbb{Q}^{\times 2}$ 
and 
$H^{1}(\mathbb{Q}_{p},E[\varphi]) \cong \mathbb{Q}_{p}^{\times}/\mathbb{Q}_{p}^{\times 2}$,  
the restriction map
\[
H^{1}(\mathbb{Q},E[\varphi])
\to 
H^{1}(\mathbb{Q}_{p},E[\varphi])\,/\,\mathcal{L}_{\varphi,p}(E)
\]
is identified with the homomorphism 
\[
\alpha \mapsto p^{\,v_{p}(\alpha)} \bmod \mathbb{Q}_{p}^{\times 2}.
\]
\end{lemma}
\begin{proof}
Since $E/\mathbb{Q}$ has good reduction at $p$, the local condition is the unramified subgroup
$\mathcal{L}_{\varphi,p}(E/\mathbb{Q})=H^{1}_{\mathrm{ur}}(\mathbb{Q}_{p},E[\varphi])$ \cite[Lemma~4.1]{Cassels1965}.  
Via the identification $H^{1}(\mathbb{Q}_{p},E[\varphi])\cong \mathbb{Q}_{p}^{\times}/\mathbb{Q}_{p}^{\times 2}$, 
the unramified subgroup identifies with the cyclic subgroup generated by a non-square unit in $\mathbb{Q}_{p}^{\times}$.  
Under the identification $H^{1}(\mathbb{Q},E[\varphi])\cong \mathbb{Q}^{\times}/\mathbb{Q}^{\times 2}$, 
the restriction map on cohomology is the homomorphism 
$\alpha \bmod \mathbb{Q}^{\times 2}\mapsto \alpha \bmod \mathbb{Q}_{p}^{\times 2}$.  
Modulo the unramified subgroup, this becomes $\alpha \mapsto p^{v_p(\alpha)} \tmod \mathbb{Q}_{p}^{\times 2}.$\qedhere
\end{proof}

\subsection{Variation of $\varphi$–Selmer groups under quadratic twisting}
\label{subsec_variaitons_in_selmer}
We now let $d$ be a squarefree integer, and write $E^{d}$ for the quadratic twist of $E/\mathbb{Q}$ by $K=\mathbb{Q}(\sqrt{d})$.  
The curve $E^{d}$ admits an affine model
$
E^{d}:\ y^{2}=x(x^{2}+a d\,x+b d^{2}),
$
and we also denote by $\varphi$ the isogeny $E^{d} \to (E^{d})'$ with kernel generated by $(0,0)$. 
Over $K$, there is an isomorphism
$$
\psi_{d}:E \to E^{d}, 
\qquad 
(x,y)\mapsto (d x, d^{3/2} y),
$$
which induces by restriction a $\mathrm{Gal}(\overline{\mathbb{Q}}/\mathbb{Q})$–equivariant isomorphism 
\begin{equation*} \label{eq:iso_on_2}
\psi_{d}:E[\varphi] \;\xrightarrow{\ \sim\ }\; E^{d}[\varphi].
\end{equation*}
Therefore, we have an isomorphism on cohomology,
\begin{equation} \label{eq:iso_on_phi}
\psi_d^{*}:H^{1}(\mathbb{Q},E[\varphi])\;\xrightarrow{\ \sim\ }\; H^{1}(\mathbb{Q},E^{d}[\varphi]).
\end{equation}
Since $E^{d}[\varphi]$ and $E[\varphi]$ are isomorphic to $\mu_{2}$ by the map $(0,0) \mapsto -1$ as $\textup{Gal}(\overline{\mathbb{Q}}/\mathbb{Q})$-modules, then 
\begin{equation} \label{eq_1}
H^{1}(\mathbb{Q},E[\varphi])
\;\stackrel{\sim}\to\; H^{1}(\mathbb{Q},\mu_2) \stackrel{\sim}\to
\mathbb{Q}^{\times}/\mathbb{Q}^{\times 2}, 
\end{equation} and similarly 
\begin{equation} \label{eq_2}
H^{1}(\mathbb{Q},E^{d}[\varphi])
\;\stackrel{\sim}\to\; H^{1}(\mathbb{Q},\mu_2) \stackrel{\sim}\to
\mathbb{Q}^{\times}/\mathbb{Q}^{\times 2}.
\end{equation}

\begin{lemma} \label{lem:identification}
Under the identifications $H^{1}(\mathbb{Q},E[\varphi])\cong \mathbb{Q}^{\times}/\mathbb{Q}^{\times 2}$
and 
$H^{1}(\mathbb{Q},E^{d}[\varphi])\cong \mathbb{Q}^{\times}/\mathbb{Q}^{\times 2}$, the induced map $\psi_{d}^{*}:H^{1}(\mathbb{Q},E[\varphi])\to H^{1}(\mathbb{Q},E^{d}[\varphi])$ acts as the identity on $\mathbb{Q}^{\times}/\mathbb{Q}^{\times 2}$. In other words, the following diagram commutes:

\begin{center}
\begin{tikzcd}
{H^{1}(\mathbb{Q},E[\varphi])} \arrow[d, "\psi_d^{*}"'] \arrow[rd, "\eqref{eq_1}"] &                                           \\
{H^{1}(\mathbb{Q}, E^{d}[\varphi])} \arrow[r, "\eqref{eq_2}"']                     & \mathbb{Q}^{\times}/\mathbb{Q}^{\times 2}
\end{tikzcd}
\end{center}
The same conclusion holds upon replacing $\mathbb{Q}$ with $\mathbb{Q}_{v}$ for any place $v$.
\end{lemma}
\begin{proof}
Fix isomorphisms $j:E[\varphi]\to\mu_{2}$ and $j_{d}:E^{d}[\varphi]\to\mu_{2}$ which, in both cases, are determined by $(0,0) \mapsto -1$. Then $j_{d}\circ\psi_{d}=j$ as homomorphisms $E[\varphi]\to\mu_{2}$.
\end{proof}

Via the identification of \eqref{eq:iso_on_phi}, we may regard the local Kummer images
\[
\mathcal{L}_{\varphi,v}(E) \subset H^{1}(\mathbb{Q}_{v},E[\varphi])
\quad\text{and}\quad
\mathcal{L}_{\varphi,v}(E^{d}) \subset H^{1}(\mathbb{Q}_{v},E^{d}[\varphi]) \cong H^{1}(\mathbb{Q}_{v},E[\varphi])
\]
as subgroups of the same cohomology group $H^{1}(\mathbb{Q}_{v}, E[\varphi])$. Similarly, we can identify the
corresponding global cohomology groups and view the global Selmer groups
$
\Sel^{\varphi}(E/\mathbb{Q}),\ \Sel^{\varphi}(E^{d}/\mathbb{Q})$ as subgroups of the same cohomology group, namely $H^{1}(\mathbb{Q},E[\varphi])$.

\subsection{Twisting by imaginary quadratic fields satisfying the Heegner hypothesis} \label{subsec_twisting}

For the remainder of this paper, we will focus on the effect on Selmer groups after twisting by $-d$, 
for $d$ belonging to the following sets.

\begin{definition} \label{def:primes}
We define
\[
\mathcal{Q}_{E/\mathbb{Q}}
=\Bigl\{\,q \text{ prime} : v \text{ splits in } \mathbb{Q}(\sqrt{-q}) 
\text{ for every } v \mid 2N_{E}\,\Bigr\},
\] 
and
\[
\mathcal{D}_{E/\mathbb{Q}}
=\Bigl\{\,d=\prod_{i=1}^{r} q_{i}\ :\ r\ge 1 \text{ odd},\ q_{i}\in \mathcal{Q}_{E/\mathbb{Q}}\ \text{distinct}\Bigr\}.
\]
\end{definition}
Equivalently, $\mathcal{Q}_{E/\mathbb{Q}}$ is the set of primes $q$ with $q \equiv 7 \tmod{8}$ and 
$\bigl(\tfrac{-q}{p}\bigr)=1$ for every odd prime $p$ where $E/\mathbb{Q}$ has bad reduction. It follows that $\mathcal{Q}_{E/\mathbb{Q}}$ is non-empty (in fact, infinite).  Similarly, for any $d\in \mathcal{D}_{E/\mathbb{Q}}$, all primes $v\mid 2N_{E}$ split in the quadratic extension $\mathbb{Q}(\sqrt{-d})/\mathbb{Q}$. In other words, the setting considered here is consistent with the \emph{Heegner hypothesis} considered in \S \ref{subsec_intro1}.
In order to understand how $\varphi$–Selmer groups vary under quadratic twisting, 
it is necessary to study how the local subgroups 
$\mathcal{L}_{\varphi,v}(E)$ change after passing to the quadratic twist $E^{-d}$. 
The following lemma shows that, away from the primes dividing $2N_{E}d\infty$
the local conditions remain unchanged.

\begin{lemma} \label{eq_local_conditions}
Let $d \in \mathcal{D}_{E/\mathbb{Q}}$.  
Under the natural identification 
$E[\varphi]\cong E^{-d}[\varphi]$, the local conditions 
$\mathcal{L}_{\varphi,v}(E)$ and $\mathcal{L}_{\varphi,v}(E^{-d})$ coincide as subgroups of 
$H^{1}(\mathbb{Q}_{v},E[\varphi])$ for every $v\mid 2N_{E}$ and every finite place $v\nmid 2dN_{E}$.
\end{lemma}

\begin{proof}
If $v\mid 2N_{E}$, then by the definition of $\mathcal{D}_{E/\mathbb{Q}}$, $v$ splits in $\mathbb{Q}(\sqrt{-d})$.  
Hence $E\cong E^{-d}$ and $E'\cong (E^{-d})'$ over $\mathbb{Q}_{v}$. The local Kummer maps for $E$ and $E^{-d}$ then identify under this isomorphism, so that $\mathcal{L}_{\varphi,v}(E)=\mathcal{L}_{\varphi,v}(E^{-d})$ whenever $v\mid 2N_E$.  
If $v\nmid 2N_{E}d\infty$, then $v$ is an odd prime and both $E$ and $E'$ have good reduction at $v$. By \cite[Lemma~4.1]{Cassels1965}, the local subgroups coincide with the unramified subgroup $H^{1}_{\mathrm{ur}}(\mathbb{Q}_{v},E[\varphi])$ and $H^{1}_{\mathrm{ur}}(\mathbb{Q}_{v},E^{-d}[\varphi])$ respectively. Since the isomorphism $H^{1}(\mathbb{Q}_{v},E^{-d}[\varphi]) \stackrel{\sim}\to H^{1}(\mathbb{Q}_{v},E[\varphi])$ of Lemma \ref{lem:identification} carries $H_{\mathrm{ur}}^{1}(\mathbb{Q}_{v},E^{-d}[\varphi])$ onto $H_{\mathrm{ur}}^{1}(\mathbb{Q}_{v},E[\varphi])$, which proves the second assertion. \qedhere
\end{proof}

In view of Lemma \ref{eq_local_conditions}, any variation in $\varphi$–Selmer groups under twisting by $-d$ can only arise from the local conditions at the places $q \mid d$ and $\infty$.

\subsection{Twisting local conditions when $v | d$}
We now turn to the local conditions at the primes dividing~$d$. 
In particular, we express the local Selmer conditions explicitly 
in terms of the signs of~$b$ and~$\Delta_{E}$ in any integral Weierstrass model for $E/\mathbb{Q}$.
\begin{lemma}\label{lem:b-Delta-model-independence}
Let $E/\Q$ be given by a Weierstrass integral model
$
y^{2}=x\bigl(x^{2}+a x+b\bigr)
$ with $a ,b \in \mathbb{Z}$.
Then, for every $q\in\mathcal{Q}_{E/\Q}$ one has
\[
\Bigl(\tfrac{b}{q}\Bigr)=
\begin{cases}
+1,& b>0,\\
-1,& b<0,
\end{cases}
\qquad\text{and}\qquad
\Bigl(\tfrac{\Delta_{E}}{q}\Bigr)=
\begin{cases}
+1,& \Delta_{E}>0,\\
-1,& \Delta_{E}<0,
\end{cases}
\]
independently of the chosen model for $E/\mathbb{Q}$.
\end{lemma}

\begin{proof}
The $2$-isogenous curve $E'$ admits an affine model of the form  
$y^2 = x(x^2 - 2ax + a^2 - 4b)$. Therefore, $16b(a^2 - 4b)^2$ is, up to a square in $\Q^{\times}$, a minimal discriminant for $E'$.  
Hence $b \equiv (\Delta_{E'})_{\textup{min}} \bmod \Q^{\times 2}$.  Since $E$ and $E'$ have the same set of bad primes, then $(\Delta_{E'})_{\textup{min}}$ is supported only on primes dividing $N_E$.  
Let $q \in \mathcal{Q}_{E/\Q}$.  
By quadratic reciprocity,  
$
( \frac{(\Delta_{E'})_{\textup{min}}}{q})
=
( \frac{\pm 1}{q} ),
$ depending on the sign of the discriminant.
Since $b$ and $(\Delta_{E'})_{\textup{min}}$ have the same sign, then $
( \frac{b}{q} )
=
( \frac{\operatorname{sgn}(b)}{q} ).
$
As $q \equiv 7 \bmod 8$, the result follows. Finally, since $\Delta_E \equiv (\Delta_E)_{\textup{min}} \bmod \Q^{\times 2}$, the same reasoning applies with $b$ replaced by $\Delta_E$.
\end{proof}

To describe the resulting change in the local conditions, we introduce the following notation.
We recall that for each place $v$ of $\mathbb{Q}$, the \emph{quadratic Hilbert symbol}
\[
(\,\cdot\,,\,\cdot\,)_{v}:\;
\mathbb{Q}_{v}^{\times}/\mathbb{Q}_{v}^{\times 2}\times
\mathbb{Q}_{v}^{\times}/\mathbb{Q}_{v}^{\times 2}
\longrightarrow \mu_{2}
\]
is the symmetric bilinear pairing defined by
$(\alpha,\beta)_{v}=1$ if and only if $\alpha$ is a norm from $\mathbb{Q}_{v}(\sqrt{\beta})$.

\begin{notation} \label{not:T_d}
Assume $b>0$ and $\Delta_{E}>0$, and fix $d\in\mathcal{D}_{E/\mathbb{Q}}$.

\begin{enumerate}
\item For every prime $q\mid d$, Lemma \ref{lem:b-Delta-model-independence} gives
$\bigl(\tfrac{b}{q}\bigr)=\bigl(\tfrac{\Delta_{E}}{q}\bigr)=1$
which in turn ensure that the polynomial
$
x^{2}-2a x+(a^{2}-4b)
$
splits over $\mathbb{Q}_{q}$.  Its roots are $q$–adic units, and we fix one of them,
denoted $\beta\in\mathbb{Z}_{q}^{\times}$.

\item We define
\[
\mathcal{T}_{d}^{+}=\bigl\{\,q \mid d:\bigl(\tfrac{d_{q}}{q}\bigr)\ne(\beta,q)_{q}\bigr\},\qquad
\mathcal{T}_{d}^{-}=\bigl\{\,q \mid d:\bigl(\tfrac{d_{q}}{q}\bigr)=(\beta,q)_{q}\bigr\},
\]
where $\bigl(\tfrac{\cdot}{q}\bigr)$ is the Legendre symbol, and $d_q = \frac{d}{q}$.
\end{enumerate}
\end{notation}

\begin{remark}\label{rem:Tdpminvariance}
We note the following:
\begin{enumerate}
\item The sets $\mathcal{T}_{d}^{\pm}$ are independent of the chosen model $y^{2}=x(x^{2}+a x+b)$ of $E/\Q$.  
Indeed, any other model of the same form satisfies $a'=u^{-2}a$, $b'=u^{-4}b$, and the corresponding root $\beta'=u^{-2}\beta$, whence $(\beta’,q)_{q}=(\beta,q)_{q}$ for every $q\mid d$ 

\item In addition, the sets $\mathcal{T}_{d}^{\pm}$ are independent of the choice of root $\beta$. Indeed, by \cite[Proposition~5.5]{Voight2013}, for any $\xi\in\mathbb{Q}_{q}^{\times}$, one has

\[
(\xi,q)_{q}=
\begin{cases}
1, & \text{if $\xi\in\mathbb{Q}_{q}^{\times 2}$},\\[2pt]
-1, & \text{otherwise.}
\end{cases}
\]
Therefore, if $\overline{\beta}$ is the other root of $x^{2}-2a x+(a^{2}-4b)$, then
$\beta\overline{\beta}=\frac{\Delta_{E}}{b^{2}}$, and since
$({\Delta_{E}},q)_q=1$ for all $q\mid d$ (see Lemma \ref{lem:b-Delta-model-independence}), then one has
$(\beta,q)_{q}=(\overline{\beta},q)_{q}$.
\end{enumerate}
\end{remark}

\begin{lemma}\label{loc_conditions_at_q}
For each prime $q\mid d \in \mathcal{D}_{E/\mathbb{Q}}$, we regard 
$\mathcal{L}_{\varphi,q}(E^{-d}/\mathbb{Q})$ as a subgroup of 
$H^{1}(\mathbb{Q}_{q},E[\varphi])$.
Then exactly one of the following holds:

\begin{enumerate}
\item[\textup{(i)}] If $b<0$, the local subgroup at $q$ is trivial:
\[
\mathcal{L}_{\varphi,q}(E^{-d}/\mathbb{Q})=\{0\}.
\]

\item[\textup{(ii)}] If $\Delta_{E}<0$ and $b>0$, the local subgroup at $q$ is maximal:
\[
\mathcal{L}_{\varphi,q}(E^{-d}/\mathbb{Q})
=H^{1}(\mathbb{Q}_{q},E[\varphi]).
\]

\item[\textup{(iii)}] If $\Delta_{E}>0$ and $b>0$, the local subgroup at $q$ is one-dimensional. In addition, under the identification $H^{1}(\mathbb{Q}_{q},E[\varphi])= \mathbb{Q}_{q}^{\times}/\mathbb{Q}_{q}^{\times 2} $, one has
\[
\mathcal{L}_{\varphi,q}(E^{-d}/\mathbb{Q})
=
\begin{cases}
\langle q\rangle,  & \text{if } q\in \mathcal{T}_{d}^{+},\\[4pt]
\langle -q\rangle, & \text{if } q\in \mathcal{T}_{d}^{-}.
\end{cases}
\]
\end{enumerate}
\end{lemma}

\begin{proof}
By \cite[Lemma 3.7]{Klagsbrun2017}, we have
\begin{equation} \label{eq:loc_cond_at_q}
\mathcal{L}_{\varphi,q}(E^{-d})
\;=\;(E^{-d})'(\mathbb{Q}_{q})[2]\big/{\varphi}\! \ (E^{-d}(\mathbb{Q}_{q})[2]).
\end{equation}
Adding on, the $2$–torsion structures are determined by the following dichotomies:
\begin{equation} \label{eq:2-tors_structures}
E^{-d}(\mathbb{Q}_{q})[2]\cong
\begin{cases}
(\mathbb{Z}/2\mathbb{Z})^{2}, & (\tfrac{\Delta_{E}}{q})=1,\\
\mathbb{Z}/2\mathbb{Z}, & (\tfrac{\Delta_{E}}{q})=-1,
\end{cases}
\qquad
(E^{-d})'(\mathbb{Q}_{q})[2]\cong
\begin{cases}
(\mathbb{Z}/2\mathbb{Z})^{2}, & (\tfrac{b}{q})=1,\\
\mathbb{Z}/2\mathbb{Z}, & (\tfrac{b}{q})=-1.
\end{cases}
\end{equation}
Combining Lemma \ref{lem:b-Delta-model-independence} with \eqref{eq:loc_cond_at_q}--\eqref{eq:2-tors_structures} gives
\[
\dim_{\mathbb{F}_{2}}\mathcal{L}_{\varphi,q}(E^{-d})=
\begin{cases}
2, & \Delta_{E}<0,\; b>0,\\
1, & \Delta_{E}>0,\; b>0,\\
0, & b<0.
\end{cases}
\]
Note that when $b<0$ one has $\Delta_{E}=b^{2}(a^{2}-4b)>0$, so the above cases are mutually exclusive and collectively exhaustive.  
This completes the proof of part~(i). Since $\dim_{\mathbb{F}_{2}}H^{1}(\mathbb{Q}_{q},E[\varphi])=2$, part~(ii) follows immediately. Now suppose $\Delta_{E}>0$ and $b>0$ and consider the affine models
\[
E' : y^{2}=x\bigl(x^{2}-2ax+(a^{2}-4b)\bigr),\qquad
(E^{-d})' : y^{2}=x\bigl(x^{2}+2ad x+d^{2}(a^{2}-4b)\bigr).
\]
Since $\bigl(\tfrac{\Delta_{E'}}{q}\bigr)=\bigl(\tfrac{b}{q}\bigr)=1$, the quadratic $x^{2}-2ax+(a^{2}-4b)$ splits over $\mathbb{Q}_{q}$ with roots
$\beta,\overline{\beta}\in\mathbb{Z}_{q}^{\times}$. Their product is $\beta\overline{\beta}=a^{2}-4b = \frac{\Delta_E}{b^2}$, from which we deduce that $\beta,\overline{\beta}$ are both units at $q$. Twisting by $-d$ sends these roots to $-d\beta$ and $-d\overline{\beta}$, and therefore
$
\mathcal{L}_{\varphi,q}(E^{-d})
$
is one–dimensional, and under the equality of \eqref{eq:loc_cond_at_q}, is generated by the class of $(-d\beta,0) \in (E^{-d})'(\mathbb{Q}_{q})[2]$. The image of this point under 
\eqref{eq:mu_phi} in $H^{1}(\mathbb{Q}_{q},E^{-d}[\varphi])\cong \mathbb{Q}_{q}^{\times}/\mathbb{Q}_{q}^{\times 2}$ 
is given by $-d\beta \bmod \mathbb{Q}_{q}^{\times 2}$.  
Since each $q\mid d$ satisfies $q\equiv 7 \tmod{8}$, then $-1$ is not a square in $\mathbb{Q}_{q}^{\times}$. As a result,
$-d\beta \equiv q \bmod \mathbb{Q}_{q}^{\times 2}$ if $q \in \mathcal{T}_{d}^{+}$,  and $-d\beta \equiv -q \bmod \mathbb{Q}_{q}^{\times 2}$ otherwise. By Lemma~\ref{lem:identification}, these subgroups remain unchanged under the isomorphism $H^{1}(\mathbb{Q}_{q},E^{-d}[\varphi]) \stackrel{\sim}\to H^{1}(\mathbb{Q}_{q},E[\varphi])$, which completes the proof of part (iii).
\end{proof}
\subsection{Twisting local conditions when $v = \infty$} We now study the local condition at the archimedean place.
In this case, its behaviour depends on the signs of $a$, $b$, and $\Delta_{E}$.

\begin{lemma} \label{lem:infinity} 
Let $E/\mathbb{Q}$ be an elliptic curve with affine model $E: y^{2}=x(x^{2}+ax+b)$.  
The local condition at the real place satisfies
\[
\mathcal{L}_{\varphi,\infty}(E)=
\begin{cases}
\{1\}, & \text{if } b<0 \ \text{or}\ \bigl(b>0,\ \Delta_{E}>0,\ a > 0\bigr),\\[4pt]
H^{1}(\mathbb{R},E[\varphi]), & \text{if } \bigl(b>0,\ \Delta_{E}<0\bigr)\ \text{or}\ \bigl(b>0,\ \Delta_{E}>0,\ a<0\bigr).
\end{cases}
\]
\end{lemma}

\begin{proof}
Since  $H^{1}(\mathbb{R},E[\varphi])\cong \mathbb{R}^{\times}/\mathbb{R}^{\times 2}$ is $1$–dimensional over $\mathbb{F}_{2}$, then $\mathcal{L}_{\varphi,\infty}(E/\mathbb{Q})$ is either $\{1\}$ or all of $H^{1}(\mathbb{R},E[\varphi])$. It thus suffices to decide whether $\varphi:E(\mathbb{R})\to E'(\mathbb{R})$ is surjective or not. This is done in \cite[\S 7.1]{DD08}:
\[
\begin{array}{|c@{\quad}c@{\quad}c |c|}
\hline
b & \Delta_{E} & a & \text{Conclusion} \\
\hline
- & + & + & \text{Surjective} \\
- & + & - & \text{Surjective} \\
\hline
+ & - & + & \text{Non-surjective} \\
+ & - & - & \text{Non-surjective} \\
+ & + & + & \text{Surjective} \\
+ & + & - & \text{Non-surjective} \\
\hline
\end{array} 
\]
This gives the cases listed in the statement. \qedhere
\end{proof}
We note that the cases listed in Lemma \ref{lem:infinity} are mutually exclusive and collectively exhaustive.

\begin{corollary} \label{cor:infinity_under_twist}
The local condition $\mathcal{L}_{\varphi,\infty}(E^{-d}/\mathbb{Q})$ at the real place satisfies
\[
\mathcal{L}_{\varphi,\infty}(E^{-d})=
\begin{cases}
\{1\}, & \text{if } b<0 \ \text{or}\ \bigl(b>0,\ \Delta_{E}>0,\ a<0\bigr),\\[4pt]
H^{1}(\mathbb{R},E[\varphi]), & \text{if } \bigl(b>0,\ \Delta_{E}<0\bigr)\ \text{or}\ \bigl(b>0,\ \Delta_{E}>0,\ a>0\bigr).
\end{cases}
\]
\end{corollary}

\begin{proof}
Twisting by $-d$ the coefficients transform as $a\mapsto -da$ and $b\mapsto bd^{2}$. The result now follows from Lemma \ref{lem:infinity}. \qedhere
\end{proof}

\section{Truncated structures and modified Selmer groups}
\label{sec_truncated}
In this section we introduce a collection of modified {Selmer structures} on $E[\varphi]$ and $E^{-d}[\varphi]$.  
These auxiliary structures provide a convenient framework for studying how Selmer groups vary under quadratic twisting.  
  
\subsection{Truncated and sign–modified Selmer structures}  \label{subsec_truncated}
We begin by recalling the relevant definitions and formulating the truncated Selmer structures that will be used in the sequel.

\begin{definition} \label{def:SelmerStructure}
Let $M$ be a finite $\mathrm{Gal}(\overline{\mathbb{Q}}/\mathbb{Q})$–module annihilated by $2$.  
Thus $M$ is a finite-dimensional $\mathbb{F}_{2}$–vector space. A \emph{Selmer structure} $\mathcal{L}=\{\mathcal{L}_{v}\}_{v}$ for $M$ is a collection of subspaces
\[
\mathcal{L}_{v} \;\subseteq\; H^{1}(\mathbb{Q}_{v},M)
\]
for each place $v$ of $\mathbb{Q}$, such that $\mathcal{L}_{v}$ is $H^{1}_{\mathrm{ur}}(\mathbb{Q}_{v},M):=
\ker\!\left( H^{1}(\mathbb{Q}_{v},M) \to H^{1}(I_{v},M)\right)$ for all but finitely many $v$, where $I_{v}\subset\mathrm{Gal}(\overline{\mathbb{Q}}_{v}/\mathbb{Q}_{v})$
denotes the inertia subgroup.  The \emph{Selmer group associated to $\mathcal{L}$} is then defined as
\[
\Sel^{\mathcal{L}}(\mathbb{Q},M)
\;=\;
\bigcap_{v}
\ker\!\Bigl( H^{1}(\mathbb{Q},M) \to H^{1}(\mathbb{Q}_{v},M)/\mathcal{L}_{v} \Bigr).
\]
\end{definition}

As a standard example, we have the usual $\varphi$–Selmer group. In particular for $E/\mathbb{Q}$ with partial $2$-torsion, then the local subgroups
$\mathcal{L}_{\varphi,v}$ of \eqref{eq:local_conditioss} form a Selmer structure for $E[\varphi]$, and the resulting Selmer group is precisely the classical $\varphi$–Selmer group $\Sel^{\varphi}(E/\mathbb{Q})$.

\begin{definition}\label{def:truncated_selmer}
Let $\mathcal{L}=\{\mathcal{L}_{v}\}_{v}$ be a Selmer structure on a finite
$\mathrm{Gal}(\overline{\mathbb{Q}}/\mathbb{Q})$–module $M$, and let $S$ be a finite set of places of~$\mathbb{Q}$.  
We define the \emph{truncated Selmer structure with respect to~$S$}, denoted
$\mathcal{L}_{S}=\{\mathcal{L}_{S,v}\}_{v}$, by
\[
\mathcal{L}_{S,v}=
\begin{cases}
H^{1}(\mathbb{Q}_{v},M), & \text{if } v\in S,\\[4pt]
\mathcal{L}_{v}, & \text{if } v\notin S.
\end{cases}
\]
We denote by $\Sel^{\mathcal{L}}_{S}(M/\mathbb{Q})$ the associated Selmer group, which satisfies
\[
\Sel^{\mathcal{L}}_{S}(\mathbb{Q}, M)
=\bigcap_{v\notin S}
\ker\!\left(
H^{1}(\mathbb{Q},M)\to
H^{1}(\mathbb{Q}_{v},M)\big/\mathcal{L}_{v}
\right).
\]
\end{definition}

The truncated Selmer group serves as an intermediate object in the computation of the usual Selmer group, obtained by temporarily omitting the local conditions at the places in~$S$.  
In particular, the usual Selmer group is obtained once the local subgroups $\mathcal{L}_{v}$ for $v\in S$ are reinstated.  
As we will see, this formulation is convenient for isolating and studying the contribution of individual primes to the Selmer group structure.

\begin{definition} \label{def:plus_selmer}
Let $\mathcal{L}=\{\mathcal{L}_{v}\}_{v}$ be a Selmer structure on a finite 
$\mathrm{Gal}(\overline{\mathbb{Q}}/\mathbb{Q})$–module $M$.  
We define the Selmer structure
$\mathcal{L}^{+}=\{\mathcal{L}^{+}_{v}\}_{v}$ by
\[
\mathcal{L}^{+}_{v}=
\begin{cases}
\mathcal{L}_{v}, & v\neq \infty,\\[4pt]
\{0\}, & v=\infty,
\end{cases}
\]
where $\{1\}$ denotes the trivial subgroup of $H^{1}(\mathbb{R},M)$.
The corresponding Selmer group is denoted
\[
\Sel^{\mathcal{L}}(\mathbb{Q},M)^{+}
\;:=\;
\Sel^{\mathcal{L}^{+}}(\mathbb{Q},M).
\]
\end{definition}

\begin{remark}
When $M=\mu_{2}$, any Selmer group $\Sel^{\mathcal{L}}(\mathbb{Q},\mu_{2})$ 
can be identified with a subspace of 
$H^{1}(\mathbb{Q},\mu_{2})\cong\mathbb{Q}^{\times}/\mathbb{Q}^{\times 2}$.
Under this identification, the modified group 
$\Sel^{\mathcal{L}}(\mathbb{Q},\mu_{2})^{+}$ is the subgroup of $\Sel^{\mathcal{L}}(\mathbb{Q},\mu_{2})$ consisting of those square classes admitting a {positive} representative in $\mathbb{Q}^{\times}$.
This interpretation motivates the superscript “$+$” in the notation.
\end{remark}

In the sequel we restrict to the case $M=E[\varphi]\cong\mu_{2}$,
so that $H^{1}(\mathbb{Q},E[\varphi])\cong\mathbb{Q}^{\times}/\mathbb{Q}^{\times 2}$. To this end, we fix the following notation.

\begin{definition} \label{def:truncated_selmer}
Let $S$ be a finite set of place for $\mathbb{Q}$. We define \emph{truncated $\varphi$-Selmer group with respect to $S$}, denoted  $\Sel^{\varphi}_{S}(E/\mathbb{Q})$, to be the Selmer group associated to the truncated Selmer structure $\mathcal{L}_{S}$ considered with $\mathcal{L}=\{\mathcal{L}_{\varphi,v}(E/\mathbb{Q})\}_v$. 

We define \emph{truncated $\varphi$-Selmer group with respect to $S$}, denoted  $\Sel^{\varphi}_{S}(E/\mathbb{Q})$, to be the truncated Selmer group 
$\Sel^{\mathcal{L}}_{S}(\mathbb{Q},E[\varphi])$ of Definition~\ref{def:truncated_selmer}, 
considered with respect to
$\mathcal{L}_{\varphi,v}(E/\mathbb{Q})$ of \eqref{eq:local_conditioss} at each place $v$ of $\mathbb{Q}$.
\end{definition}

To describe and compare the local conditions defining the $\varphi$–Selmer groups of $E$ and its quadratic twists, 
we introduce $\mathbb{F}_{2}$–valued homomorphisms on $\mathbb{Q}^{\times}/\mathbb{Q}^{\times 2}$ 
that record how adjoining a local condition affects the truncated Selmer group (see Lemma~\ref{lem:local_condition_cases}).

\begin{definition} \label{def:homomorphisms}
We define homomorphisms 
$\mathbb{Q}^{\times}/\mathbb{Q}^{\times 2} \to \mathbb{F}_{2}=\{0,1\}$ 
as follows:
\begin{enumerate}
\item The \emph{sign homomorphism} at $v=\infty$:
\[
\varepsilon_{\infty}(\alpha)=
\begin{cases}
0, & \alpha>0,\\
1, & \alpha<0.
\end{cases}
\]

\item For a finite prime $p$, the \emph{valuation parity homomorphism}
\[
\lambda_{p}(\alpha)= v_{p}(\alpha)\bmod 2.
\]

\item For a finite prime $p$, the \emph{quadratic residue homomorphism}: 
writing $\alpha=u_{\alpha}\,p^{v_{p}(\alpha)}$ with $p\nmid u_{\alpha}$, set
\[
\gamma_{p}(\alpha)=
\begin{cases}
0, & \bigl(\tfrac{u_{\alpha}}{p}\bigr)=1,\\
1, & \bigl(\tfrac{u_{\alpha}}{p}\bigr)=-1,
\end{cases}
\]
where $\bigl(\tfrac{\cdot}{p}\bigr)$ is the Legendre symbol.
\end{enumerate}
\end{definition}

We note that  truncated Selmer groups of $E/\mathbb{Q}$ are naturally subgroups of $\mathbb{Q}^{\times}/\mathbb{Q}^{\times 2}$. Accordingly, operations such as intersection with 
$\ker(\lambda_{p})$, $\ker(\gamma_{p})$ 
are to be interpreted using this identification.

\begin{lemma}\label{lem:local_condition_cases}
Let $\mathcal{T}$ be a finite set of places of $\mathbb{Q}$ containing an odd prime $p$.
Then, according to the local condition $\mathcal{L}_{\varphi,p}\subseteq H^{1}(\mathbb{Q}_{p},E[\varphi])$,
one has the following descriptions of the truncated Selmer group:
\begin{enumerate}
\item[(i)] If $\mathcal{L}_{\varphi,p}=H^{1}(\mathbb{Q}_{p},E[\varphi])$, then
\[
\Sel^{\varphi}_{\mathcal{T}\setminus\{p\}}(E/\mathbb{Q})
=\Sel^{\varphi}_{\mathcal{T}}(E/\mathbb{Q}).
\]

\item[(ii)] If $\mathcal{L}_{\varphi,p}=\{0\}$, then
\[
\Sel^{\varphi}_{\mathcal{T}\setminus\{p\}}(E/\mathbb{Q})
=\Sel^{\varphi}_{\mathcal{T}}(E/\mathbb{Q})
\cap\ker(\lambda_{p})
\cap \ker(\gamma_{p}).
\]

\item[(iii)] If $\mathcal{L}_{\varphi,p}=H^{1}_{\mathrm{ur}}(\mathbb{Q}_{p},E[\varphi])$, then
\[
\Sel^{\varphi}_{\mathcal{T}\setminus\{p\}}(E/\mathbb{Q})
=\Sel^{\varphi}_{\mathcal{T}}(E/\mathbb{Q})
\cap\ker(\lambda_{p}).
\]

\item[(iv)] If, under the identification
$H^{1}(\mathbb{Q}_{p},E[\varphi])\cong\mathbb{Q}_{p}^{\times}/\mathbb{Q}_{p}^{\times 2}$, $\mathcal{L}_{\varphi,p}$ is the subgroup generated by the class of $p \tmod \mathbb{Q}_{p}^{\times 2}$, then
\[
\Sel^{\varphi}_{\mathcal{T}\setminus\{p\}}(E/\mathbb{Q})
=\Sel^{\varphi}_{\mathcal{T}}(E/\mathbb{Q})
\cap\ker(\gamma_{p}).
\]

\item[(v)] If, under the same identification, $\mathcal{L}_{\varphi,p}$ is the subgroup generated by the class of $p u \tmod\mathbb{Q}_{p}^{\times 2}$, for some non-square unit $u\in\mathbb{Z}_{p}^{\times}$, then
\[
\Sel^{\varphi}_{\mathcal{T}\setminus\{p\}}(E/\mathbb{Q})
=\Sel^{\varphi}_{\mathcal{T}}(E/\mathbb{Q})
\cap\ker(\gamma_{p}+\lambda_{p}).
\]
\end{enumerate}
\end{lemma}

\begin{proof}
By definition of truncated Selmer groups,
\[
\Sel^{\varphi}_{\mathcal{T}\setminus\{p\}}(E/\mathbb{Q})
= \Sel^{\varphi}_{\mathcal{T}}(E/\mathbb{Q}) \cap 
\ker\!\left(H^{1}(\mathbb{Q},E[\varphi])\to
H^{1}(\mathbb{Q}_{p},E[\varphi])/\mathcal{L}_{\varphi,p}\right).
\]
Thus passing from $\Sel^{\varphi}_{\mathcal{T}}$ to $\Sel^{\varphi}_{\mathcal{T}\setminus\{p\}}$ amounts to
imposing the additional condition that the localization at $p$ maps to $\mathcal{L}_{\varphi,p}$. Identify $H^{1}(\mathbb{Q},E[\varphi])\cong\mathbb{Q}^{\times}/\mathbb{Q}^{\times 2}$ and
$H^{1}(\mathbb{Q}_{p},E[\varphi])\cong\mathbb{Q}_{p}^{\times}/\mathbb{Q}_{p}^{\times 2}$, and write
for $\alpha\in\mathbb{Q}^{\times}/\mathbb{Q}^{\times 2}$ its $p$-adic factorization
$\alpha=p^{v_{p}(\alpha)}u_{\alpha}$ with $p \nmid u_{\alpha}.$
(i) If $\mathcal{L}_{\varphi,p}=H^{1}(\mathbb{Q}_{p},E[\varphi])$, then there is no new condition at $p$ and
$\Sel^{\varphi}_{\mathcal{T}\setminus\{p\}}=\Sel^{\varphi}_{\mathcal{T}}$. (ii) If $\mathcal{L}_{\varphi,p}=\{0\}$, the extra condition is that the $p$-local class is trivial.
Equivalently $p^{\lambda_{p}(\alpha)}u_{\alpha}\in\mathbb{Q}_{p}^{\times 2}$, i.e.
$\lambda_{p}(\alpha)=0$ and $\big(\frac{u_{\alpha}}{p}\big)=+1$. (iii) If $\mathcal{L}_{\varphi,p}=H^{1}_{\mathrm{ur}}(\mathbb{Q}_{p},E[\varphi])$,
then the $p$-local image is unramified, which happens only when $\lambda_{p}(\alpha)=0$. (iv) If $\mathcal{L}_{\varphi,p}=\langle p\rangle\subset\mathbb{Q}_{p}^{\times}/\mathbb{Q}_{p}^{\times 2}$,
the $p$-local class lies in $\langle p\rangle$ only when $u_{\alpha}\in\mathbb{Q}_{p}^{\times 2}$, i.e.
$\big(\frac{u_{\alpha}}{p}\big)=+1$. (v) If $\mathcal{L}_{\varphi,p}=\langle p\cdot u\rangle$ with $u\in\mathbb{Z}_{p}^{\times}$ a fixed nonsquare, then the $p$-local class lies in $\langle p\cdot u\rangle$ if and only if $\lambda_p(\alpha)=1$ when $(\tfrac{u_{\alpha}}{p})=-1$ or when $\lambda_p(\alpha)=0$ when $(\tfrac{u_{\alpha}}{p})=1$. Equivalently, $\gamma_{p}(\alpha)+\lambda_{p}(\alpha)=0$ in $\mathbb{F}_{2}$.
\end{proof}

\begin{lemma}\label{lem:local_conditions_infinity}
Let $\mathcal{T}$ be a finite set of places of $\mathbb{Q}$ containing the Archimedean place~$\infty$.
Then, according to the local condition 
$\mathcal{L}_{\varphi,\infty}\subseteq H^{1}(\mathbb{R},E[\varphi])$, 
the truncated Selmer group satisfies:
\begin{enumerate}
\item[(i)] If $\mathcal{L}_{\varphi,\infty}=H^{1}(\mathbb{R},E[\varphi])$, then
\[
\Sel^{\varphi}_{\mathcal{T}\setminus\{\infty\}}(E/\mathbb{Q})
=\Sel^{\varphi}_{\mathcal{T}}(E/\mathbb{Q}).
\]

\item[(ii)] If $\mathcal{L}_{\varphi,\infty}=\{0\}$, then
\[
\Sel^{\varphi}_{\mathcal{T}\setminus\{\infty\}}(E/\mathbb{Q})
=\Sel^{\varphi}_{\mathcal{T}}(E/\mathbb{Q})^{+},
\]
where the superscript $``{+}"$ is as in Definition~\ref{def:plus_selmer}.
\end{enumerate}
\end{lemma}
\begin{proof}
By the definition of truncated Selmer groups,
\[
\Sel^{\varphi}_{\mathcal{T}\setminus\{\infty\}}(E/\mathbb{Q})
=\Sel^{\varphi}_{\mathcal{T}}(E/\mathbb{Q}) \cap 
\ker\!\left(H^{1}(\mathbb{Q},E[\varphi])
\to
H^{1}(\mathbb{R},E[\varphi])/\mathcal{L}_{\varphi,\infty}\right).
\]
(i) If $\mathcal{L}_{\varphi,\infty}=H^{1}(\mathbb{R},E[\varphi])$, no additional local constraint is imposed, and hence 
$\Sel^{\varphi}_{\mathcal{T}\setminus\{\infty\}}(E/\mathbb{Q})
=\Sel^{\varphi}_{\mathcal{T}}(E/\mathbb{Q})$. (ii) If $\mathcal{L}_{\varphi,\infty}=\{0\}$, the result follows by definition.
\end{proof}

\subsection{Linear–algebraic description of Selmer conditions}
\label{subsec_linear_algebra}
In this subsection, we study the dimensions of subspaces of $\mathbb{Q}^{\times}/\mathbb{Q}^{\times 2}$ 
intersected with the kernels of the homomorphisms introduced in Definition~\ref{def:homomorphisms}.  
In the sequel, this will allow us to reduce the study of how Selmer ranks vary under quadratic twisting to a problem in linear algebra over~$\mathbb{F}_{2}$.

\begin{notation}\label{not:GammaLambda}
Let $V\subset\mathbb{Q}^{\times}/\mathbb{Q}^{\times 2}$ be a finite–dimensional $\mathbb{F}_{2}$–subspace.
We call a finite set $\Sigma$ a \emph{supporting set} for~$V$ if
\[
 \bigl\{\,p \text{ non-archimedean prime} : 
\lambda_{p}(\alpha)\ne 0 \text{ for some }\alpha\in V\,\bigr\} \subseteq \Sigma.
\]
With such a choice of~$\Sigma$, we make the following constructions:

\begin{enumerate}
\item[(i)] We define an injective map
\[
\iota:V\to\mathbb{F}_{2}^{\Sigma\cup\{\infty\}},\qquad
\alpha\mapsto\bigl(\varepsilon_{\infty}(\alpha),(\lambda_{p}(\alpha))_{p\in\Sigma}\bigr).
\]

\item[(ii)] For a finite prime $q$ (not necessarily in $\Sigma$), we introduce the row vectors
\[
\mathbf{e}^{(q)}=(e^{(q)}_{\infty},(e^{(q)}_{p})_{p\in\Sigma}),
\qquad
\mathbf{c}^{(q)}=(c^{(q)}_{\infty},(c^{(q)}_{p})_{p\in\Sigma})
\in\mathbb{F}_{2}^{\Sigma\cup\{\infty\}},
\]
whose entries are given by
\begin{align*}
e^{(q)}_{\infty} &= 0, &
c^{(q)}_{\infty} &=
\begin{cases}
0, & (\tfrac{-1}{q})=+1,\\[2pt]
1, & (\tfrac{-1}{q})=-1,
\end{cases}\\[6pt]
e^{(q)}_{p} &=
\begin{cases}
1, & p=q,\\[2pt]
0, & p\ne q,
\end{cases}
&
c^{(q)}_{p} &=
\begin{cases}
0, & p=q\text{ or }(\tfrac{p}{q})=+1,\\[2pt]
1, & (\tfrac{p}{q})=-1.
\end{cases}
\end{align*}
\end{enumerate}

\end{notation}

\begin{definition}\label{def:matrices}
Let $\iota:V\to\mathbb{F}_{2}^{\Sigma\cup\{\infty\}}$ be as in Notation~\ref{not:GammaLambda}. Then, we define the following.
\begin{enumerate}
\item[(i)] If $\mathcal{T}$ is a finite set of non-archimedean primes, we define:
\begin{itemize}[leftmargin=*]
\item the \emph{$\gamma$–matrix} associated to $\mathcal{T}$ and $V$ as
\[
\Gamma_{\mathcal{T}}
:= 
\begin{pmatrix}
(\mathbf{c}^{(q)})_{q\in\mathcal{T}}
\end{pmatrix}
\in
\mathrm{Mat}_{\,|\mathcal{T}|\times(|\Sigma|+1)}(\mathbb{F}_{2}),
\]
and write $\Gamma_{\mathcal{T}}\!\mid_{V}$ to denote the restriction to $\iota(V)$;
\item the \emph{$\lambda$–matrix} associated to $\mathcal{T}$ and $V$ as
\[
\Lambda_{\mathcal{T}}
:=
\begin{pmatrix}
(\mathbf{e}^{(q)})_{q\in\mathcal{T}}
\end{pmatrix}
\in
\mathrm{Mat}_{\,|\mathcal{T}|\times(|\Sigma|+1)}(\mathbb{F}_{2}),
\]
and write $\Lambda_{\mathcal{T}}\!\mid_{V}$ to denote the restriction to $\iota(V)$.
\end{itemize}

\item[(ii)] If $\mathcal{T}=\{\infty\}$ consists of the real place, we define the \emph{sign matrix}
\[
S_{\infty}:=(1,0,\ldots,0)\in
\mathrm{Mat}_{1\times(|\Sigma|+1)}(\mathbb{F}_{2}),
\]
and write $S_{\infty}|V$ to denote its restriction to $\iota(V)$.
\end{enumerate}
\end{definition}

\begin{lemma}\label{lem:gamma-lambda-matrix}
Let $V\subset\mathbb{Q}^{\times}/\mathbb{Q}^{\times 2}$ be a finite–dimensional $\mathbb{F}_{2}$–subspace supported on~$\Sigma$, and let $\mathcal{T}$ be a finite set of non-archimedean primes.  
Then, irrespective of the  choice of supporting set $\Sigma$ for $V$, one has
\[
\begin{aligned}
V\cap \bigcap_{q\in\mathcal{T}}\ker(\gamma_{q})
&=\ker\!\bigl(\Gamma_{\mathcal{T}}\!\mid_{V}\bigr),\\[6pt]
V\cap \bigcap_{q\in\mathcal{T}}\ker(\lambda_{q})
&=\ker\!\bigl(\Lambda_{\mathcal{T}}\!\mid_{V}\bigr),\\[6pt]
V\cap \ker(\varepsilon_{\infty})
&=\ker\!\bigl(S_{\infty}\!\mid_{V}\bigr).
\end{aligned}
\]
\end{lemma}
                   
\begin{proof}
For $\alpha\in V$ and $q\in\mathcal{T}$, write $\alpha=q^{v_q(\alpha)}u_{\alpha}$ with $q\nmid u_{\alpha}$.  
Then $\gamma_q(\alpha)=0$ if and only if $\bigl(\tfrac{u_{\alpha}}{q}\bigr)=1$, that is, 
$\langle\mathbf{c}^{(q)},\iota(\alpha)\rangle=0$; hence
\[
\alpha\in V\cap\!\bigcap_{q\in\mathcal{T}}\ker(\gamma_q)
\iff
\Gamma_{\mathcal{T}}\,\iota(\alpha)=0
\iff
\iota(\alpha)\in\textup{ker}(\Gamma_{\mathcal{T}}). \qedhere
\]  
\end{proof}

\section{Variation in Selmer groups after quadratic twists}
In this section we describe explicitly the $\varphi$–Selmer groups of $E$ and of its quadratic twists, using the framework of truncated Selmer groups introduced in \S\ref{subsec_truncated} together with the homomorphisms of Definition~\ref{def:homomorphisms}.  Theorem~\ref{thm:master} below summarises all possible cases, according to the signs of $a$, $b$, and $\Delta_{E}$.

\subsection{Structure variation under twisting}
We begin by fixing some notation.

\begin{notation} \label{not:primesofd}
For $d \in \mathcal{D}_{E/\mathbb{Q}}$, set
\[
\mathcal{S}_{d} = \{\, v : v \text{ is a place of } \mathbb{Q},\ v \mid d\infty \,\},
\qquad
\mathcal{S}_{d}^{0} = \{\, v : v \text{ is a finite prime of } \mathbb{Q},\ v \mid d \,\}.
\]
\end{notation}

In what follows, the subsets $\mathcal{T}_{d}^{+}$ and $\mathcal{T}_{d}^{-}$ of $\mathcal{S}_{d}^{0}$ 
are those introduced in Notation~\ref{not:T_d}.
In addition, truncated Selmer groups of $E/\mathbb{Q}$ are naturally subgroups of $\mathbb{Q}^{\times}/\mathbb{Q}^{\times 2}$, 
and hence operations such as intersections with 
$\ker(\lambda_{p})$ or $\ker(\gamma_{p})$ 
are interpreted in this setting.

\begin{theorem}\label{thm:master}
Let $E/\mathbb{Q}$ be given by the affine model
$
E:\ y^{2}=x\bigl(x^{2}+a x+b\bigr),
$
and let $d\in\mathcal{D}_{E/\mathbb{Q}}$.  
Then, as subgroups of $H^{1}(\mathbb{Q},E[\varphi])$, one has:
\begin{enumerate}
\item[(i)] If $b<0$ (and hence $\Delta_{E}>0$), then
\[
\Sel^{\varphi}(E/\mathbb{Q})
=\Sel^{\varphi}(E^{-d}/\mathbb{Q}).
\]
\item[(ii)] If $\Delta_E<0$ and $b>0$, then
\begin{align*}
\Sel^{\varphi}(E/\mathbb{Q}) &\;=\; 
\Sel^{\varphi}_{\mathcal{S}_{d}}(E/\mathbb{Q}) \cap \bigcap_{q|d} \ker(\lambda_{q}), \\
\Sel^{\varphi}(E^{-d}/\mathbb{Q}) &\;=\; 
\Sel^{\varphi}_{\mathcal{S}_{d}}(E/\mathbb{Q}).
\end{align*}

\item[(iii)] If $\Delta_{E}>0$, $b>0$, and $a<0$, then
\begin{align*}
\Sel^{\varphi}(E/\mathbb{Q})
&=\Sel^{\varphi}_{\mathcal{S}_{d}}(E/\mathbb{Q})
   \cap \bigcap_{q\mid d}\ker(\lambda_{q}),\\[4pt]
\Sel^{\varphi}(E^{-d}/\mathbb{Q})
&=\Sel^{\varphi}_{\mathcal{S}_{d}}(E/\mathbb{Q})^{+}
   \cap
   \Bigl(\bigcap_{q\in\mathcal{T}_{d}^{+}}\ker(\gamma_{q})\Bigr)
   \cap
   \Bigl(\bigcap_{q\in\mathcal{T}_{d}^{-}}\ker(\gamma_{q}+\lambda_{q})\Bigr).
\end{align*}

\item[(iv)] If $\Delta_{E}>0$, $b>0, a>0$, then
\begin{align*}
\Sel^{\varphi}(E/\mathbb{Q})
&=\Sel^{\varphi}_{\mathcal{S}_{d}}(E/\mathbb{Q})^{+} \cap \bigcap_{q|d} \ker(\lambda_{q}), \\
\Sel^{\varphi}(E^{-d}/\mathbb{Q})&=\Sel^{\varphi}_{\mathcal{S}_{d}}(E/\mathbb{Q})
   \cap
   \Bigl(\bigcap_{q\in\mathcal{T}_{d}^{+}}\ker(\gamma_{q})\Bigr)
   \cap
   \Bigl(\bigcap_{q\in\mathcal{T}_{d}^{-}}\ker(\gamma_{q}+\lambda_{q})\Bigr).
\end{align*}
\end{enumerate}
\end{theorem}
Before turning to the proof of this theorem, we present some immediate corollaries.  
Since $\Sel^{\varphi}(E/\Q)$ and $\Sel^{\varphi}(E^{-d}/\Q)$ lie in the same ambient group $H^{1}(\Q,E[\varphi])$, their intersection is well defined. The following corollary describes these intersections explicitly.
\begin{corollary}\label{cor:all_intersections}
In the situation of Theorem \ref{thm:master}, and as subgroups of $H^{1}(\mathbb{Q},E[\varphi])$, one has:
\begin{enumerate}
\item[(i)] If $b<0$ (hence $\Delta_{E}>0$), then
\[
\Sel^{\varphi}(E/\mathbb{Q})\ \cap\ \Sel^{\varphi}(E^{-d}/\mathbb{Q})
\;=\;
\Sel^{\varphi}(E/\mathbb{Q})
\;=\;
\Sel^{\varphi}(E^{-d}/\mathbb{Q}).
\]

\item[(ii)] If $\Delta_{E}<0$ and $b>0$, then
\[
\Sel^{\varphi}(E/\mathbb{Q})\ \cap\ \Sel^{\varphi}(E^{-d}/\mathbb{Q})
\;=\;
\Sel^{\varphi}_{\mathcal{S}_{d}}(E/\mathbb{Q}) \cap \bigcap_{q\mid d}\ker(\lambda_{q})
\;=\;
\Sel^{\varphi}(E/\mathbb{Q}).
\]

\item[(iii)] If $\Delta_{E}>0$, $b>0$, and $a<0$, then
\[
\Sel^{\varphi}(E/\mathbb{Q})\ \cap\ \Sel^{\varphi}(E^{-d}/\mathbb{Q})
\;=\;
\Sel^{\varphi}_{\mathcal{S}_{d}}(E/\mathbb{Q})^{+}
\ \cap\ \bigcap_{q\mid d}\ker(\lambda_{q}).
\]

\item[(iv)] If $\Delta_{E}>0$, $b>0$, and $a>0$, then
\[
\Sel^{\varphi}(E/\mathbb{Q})\ \cap\ \Sel^{\varphi}(E^{-d}/\mathbb{Q})
\;=\;
\Sel^{\varphi}_{\mathcal{S}_{d}}(E/\mathbb{Q})^{+}
\ \cap\ \bigcap_{q\mid d}\ker(\lambda_{q}).
\]
\end{enumerate}
\end{corollary}

\begin{proof}
Using the descriptions in Theorem~\ref{thm:master}, parts \textup{(i)}–\textup{(ii)} are immediate.  
In case \textup{(iii)}, the intersection is given by
\[   \textup{Sel}_{\mathcal{S}_{d}}^{\varphi}(E/\mathbb{Q})^{+} \cap \Bigr{(} \bigcap_{q|d} \textup{ker}(\lambda_q) \Bigr{)}  \cap
   \Bigl(\bigcap_{q\in\mathcal{T}_{d}^{+}}\ker(\gamma_{q})\Bigr)
   \cap
   \Bigl(\bigcap_{q\in\mathcal{T}_{d}^{-}}\ker(\gamma_{q}+\lambda_{q})\Bigr).\]
By Corollary~\ref{cor:gamma_trivial} below, each $\gamma_{q}$ in this description vanishes on 
$\Sel^{\varphi}_{\mathcal{S}_{d}}(E/\mathbb{Q})^{+}\cap\bigcap_{q\mid d}\ker(\lambda_{q})$, giving the required description in case (iii). The result in case (iv) follows similarly. \qedhere
\end{proof}

\begin{corollary}\label{cor:master-Eprime}
Let $ E':\ y^{2}=x(x^{2}-2a x+a^2-4b)$
be the $2$–isogenous curve with dual isogeny $\hat{\varphi}:E'\to E$. Then, for each $d\in\mathcal{D}_{E/\mathbb{Q}}$, one has:
\begin{enumerate}
\item[(i)] If $b<0$ (and hence $\Delta_{E}>0$), then
\begin{align*}
\Sel^{\hat{\varphi}}(E'/\mathbb{Q})
&=\Sel^{\hat{\varphi}}_{\mathcal{S}_{d}}(E'/\mathbb{Q})\cap\bigcap_{q\mid d}\ker(\lambda_{q}), \\
\Sel^{\hat{\varphi}}(E'^{-d}/\mathbb{Q})
&=\Sel^{\hat{\varphi}}_{\mathcal{S}_{d}}(E'/\mathbb{Q}), \\
\Sel^{\hat{\varphi}}(E'/\mathbb{Q})
\cap
\Sel^{\hat{\varphi}}(E'^{-d}/\mathbb{Q})
&=\Sel^{\hat{\varphi}}(E'/\mathbb{Q}).
\end{align*}

\medskip

\item[(ii)] If $\Delta_{E}<0$ and $b>0$, then
\begin{align*}
\Sel^{\hat{\varphi}}(E'/\mathbb{Q})
&=\Sel^{\hat{\varphi}}(E'^{-d}/\mathbb{Q}), \\
\Sel^{\hat{\varphi}}(E'/\mathbb{Q})
\cap
\Sel^{\hat{\varphi}}(E'^{-d}/\mathbb{Q})
&=\Sel^{\hat{\varphi}}(E'/\mathbb{Q}).
\end{align*}

\medskip

\item[(iii)] If $\Delta_{E}>0$, $b>0$, and $a<0$, then
\begin{align*}
\Sel^{\hat{\varphi}}(E'/\mathbb{Q})
&=\Sel^{\hat{\varphi}}_{\mathcal{S}_{d}}(E'/\mathbb{Q})^{+}
   \cap\bigcap_{q\mid d}\ker(\lambda_{q}),\\[4pt]
\Sel^{\hat{\varphi}}(E'^{-d}/\mathbb{Q})
&=\Sel^{\hat{\varphi}}_{\mathcal{S}_{d}}(E'/\mathbb{Q})
   \cap
   \Bigl(\bigcap_{q\in\mathcal{T}_{d}^{+}}\ker(\gamma_{q})\Bigr)
   \cap
   \Bigl(\bigcap_{q\in\mathcal{T}_{d}^{-}}\ker(\gamma_{q}+\lambda_{q})\Bigr),\\
   \Sel^{\hat{\varphi}}(E'/\mathbb{Q})
\cap
\Sel^{\hat{\varphi}}(E'^{-d}/\mathbb{Q})
&=\Sel^{\hat{\varphi}}_{\mathcal{S}_{d}}(E'/\mathbb{Q})^{+}
\cap
\bigcap_{q\mid d}\ker(\lambda_{q}).
\end{align*}

\medskip

\item[(iv)] If $\Delta_{E}>0$, $b>0$, and $a>0$, then
\begin{align*}
\Sel^{\hat{\varphi}}(E'/\mathbb{Q})
&=\Sel^{\hat{\varphi}}_{\mathcal{S}_{d}}(E'/\mathbb{Q})
   \cap\bigcap_{q\mid d}\ker(\lambda_{q}),\\[4pt]
\Sel^{\hat{\varphi}}(E'^{-d}/\mathbb{Q})
&=\Sel^{\hat{\varphi}}_{\mathcal{S}_{d}}(E'/\mathbb{Q})^{+}
   \cap
   \Bigl(\bigcap_{q\in\mathcal{T}_{d}^{+}}\ker(\gamma_{q})\Bigr)
   \cap
   \Bigl(\bigcap_{q\in\mathcal{T}_{d}^{-}}\ker(\gamma_{q}+\lambda_{q})\Bigr),\\
   \Sel^{\hat{\varphi}}(E'/\mathbb{Q})
\cap
\Sel^{\hat{\varphi}}(E'^{-d}/\mathbb{Q})
&=\Sel^{\hat{\varphi}}_{\mathcal{S}_{d}}(E'/\mathbb{Q})
\cap
\bigcap_{q\mid d}\ker(\lambda_{q}).
\end{align*}
\end{enumerate}
\end{corollary}
\begin{proof} 
Since $E/\mathbb{Q}$ and $E'/\mathbb{Q}$ are isogenous, then they have the same set of bad primes, and therefore $\mathcal{D}_{E/\mathbb{Q}}= \mathcal{D}_{E'/\mathbb{Q}}$. As a result, we may apply Theorem~\ref{thm:master} to $E'$ for every $d\in\mathcal{D}_{E/\mathbb{Q}}$. In this case however, $(a',b',\Delta_{E'})=(-2a,\ \Delta_{E}\cdot (\textup{square}),\ b\cdot(\textup{square}))$,  
so in the sign configuration of Theorem~\ref{thm:master},  
the roles of $b$ and $\Delta_{E}$ are interchanged and the sign of $a$ is reversed. This gives the claimed description for the Selmer groups. For the claim regarding their intersections, we argue as in Corollary \ref{cor:all_intersections}. 
\end{proof}
\subsection{Proof of Theorem \ref{thm:master}}
We first record a few auxiliary results for the truncated Selmer group $\textup{Sel}_{\mathcal{S}_{d}}^{\varphi}(E/\mathbb{Q})$ that will be necessary for the arguments that follow.

\begin{lemma} \label{lem:eq_local_conditions}
Let $d \in \mathcal{D}_{E/\mathbb{Q}}$. Then, viewed as subgroups of $H^{1}(\mathbb{Q},E[\varphi])$, 
\[
\Sel^{\varphi}_{\mathcal{S}_d}(E^{-d}/\mathbb{Q})
\;=\;
\Sel^{\varphi}_{\mathcal{S}_d}(E/\mathbb{Q}).
\]
Moreover, viewed as a subgroup of $\mathbb{Q}^{\times}/\mathbb{Q}^{\times 2}$, every element of this group admits a representative as a square–free integer supported only on primes dividing $2dN_{E}$.
\end{lemma}
\begin{proof}
By Lemma~\ref{eq_local_conditions}, the local conditions for $E$ and $E^{-d}$ coincide at every finite place $v \nmid 2dN_{E}$ and at every $v \mid 2N_{E}$. This gives the equality $\Sel^{\varphi}_{\mathcal{S}_d}(E^{-d}/\mathbb{Q})
\;=\;
\Sel^{\varphi}_{\mathcal{S}_d}(E/\mathbb{Q})$. For the second assertion, set $\mathcal{R}:= \mathcal{S}_{d} \cup \{v \mid 2N_E\}$. By definition, $\textup{Sel}_{\mathcal{S}_{d}}^{\varphi}(E/\mathbb{Q}) \subseteq  \textup{Sel}_{\mathcal{R}}^{\varphi}(E/\mathbb{Q})$. To compute $\textup{Sel}_{\mathcal{R}}^{\varphi}(E/\mathbb{Q})$, we only need to consider the local conditions at $v$ not dividing $2dN_{E}{\infty}$.  Since $E$ has good reduction at all odd primes $v \nmid 2dN_E$, then Lemma~\ref{lem:unramified_condition} shows that an element $\alpha$ in $\Sel^{\varphi}_{\mathcal{R}}(E/\mathbb{Q})$ satisfies $v(\alpha) \equiv 0 \tmod 2$ for all $v \not\in\mathcal{R}$.
\end{proof}

\begin{corollary} \label{cor:gamma_trivial} 
For any prime $\tilde{q}\in\mathcal{Q}_{E/\mathbb{Q}}$, the character $\gamma_{\tilde{q}}$ of Definition \ref{def:homomorphisms} is trivial on
$\Sel^{\varphi}_{\mathcal{S}_{d}}(E/\mathbb{Q})^{+} \cap \bigcap_{q|d} \textup{ker}(\lambda_q)$.  
\end{corollary}

\begin{proof}
In view of Lemma~\ref{lem:eq_local_conditions}, any 
$\alpha\in \Sel^{\varphi}_{\mathcal{S}_{d}}(E/\mathbb{Q})^{+} \cap \bigcap_{q|d} \textup{ker}(\lambda_q)$
admits a positive, square–free representative supported only on primes dividing $2N_{E}$. From this, we deduce that $v_{\tilde{q}}(\alpha)=0$. For $\tilde{q}\in\mathcal{Q}_{E/\mathbb{Q}}$ one has $(\tfrac{2}{\tilde{q}})=1$, and for every odd $p\mid N_{E}$ the condition
$(\tfrac{-\tilde{q}}{p})=1$ implies $(\tfrac{p}{\tilde{q}})=1$ by quadratic reciprocity. Hence  $(\tfrac{\alpha}{\tilde{q}})=1$.
\end{proof}

\begin{lemma} \label{lem:support}
Let $d \in \mathcal{D}_{E/\mathbb{Q}}$, and write $\Sigma$ for a supporting set for 
$
\Sel^{\varphi}_{\mathcal{S}_d}(E^{-d}/\mathbb{Q})
=\Sel^{\varphi}_{\mathcal{S}_d}(E/\mathbb{Q})
$
as in Notation~\ref{not:GammaLambda}. Then $\mathcal{S}_{d}^{0} \subseteq \Sigma$.
\end{lemma}
\begin{proof}
Let $q\mid d$, and view $\textup{Sel}_{\mathcal{S}_{d}}^{\varphi}(E/\mathbb{Q})$ as a subgroup of $\mathbb{Q}^{\times}/\mathbb{Q}^{\times }$. 
It suffices to show that the class of $-q\bmod\mathbb{Q}^{\times2}$ lies in 
$\Sel^{\varphi}_{\mathcal{S}_d}(E/\mathbb{Q})$. This is tantamount to the image of $-q$ in 
$H^{1}(\mathbb{Q}_{p},E[\varphi])\cong\mathbb{Q}_{p}^{\times}/\mathbb{Q}_{p}^{\times2}$
landing in $\mathcal{L}_{\varphi,p}(E)$ under localization for all $p\nmid d\infty$. If $p\nmid 2dN_E \infty$, then $E$ has good reduction and 
$\mathcal{L}_{\varphi,p}(E)=H^{1}_{\mathrm{ur}}(\mathbb{Q}_{p},E[\varphi])$.
Since $v_p(-q)=0$, then $-q$ satisfies the local condition at $ p \nmid 2N_E$ (see Lemma~\ref{lem:unramified_condition}).
If $p\mid 2N_E$, then by the defining property of $\mathcal{D}_{E/\mathbb{Q}}$, 
$-q$ is a square in $\mathbb{Q}_{p}^{\times}$, 
hence its class in $\mathbb{Q}_{p}^{\times}/\mathbb{Q}_p^{\times 2}$ is trivial.
\end{proof}

As seen by Lemma \ref{lem:eq_local_conditions}, any difference between $\Sel^{\varphi}(E/\mathbb{Q})$ and $\Sel^{\varphi}(E^{-d}/\mathbb{Q})$ can only arise from the local conditions at $q\mid d$ and at $\infty$.  
By Lemma~\ref{loc_conditions_at_q} and Corollary~\ref{cor:infinity_under_twist}, these local conditions are governed by the signs of $b$, $\Delta_{E}$, and $a$.  
In what follows we analyse separately the four possible situations:
\[
\text{(i) } b<0;\quad
\text{(ii) } b>0,\ \Delta_{E}<0;\quad
\text{(iii) } b>0,\ \Delta_{E}>0,\ a>0;\quad
\text{(iv) } b>0,\ \Delta_{E}>0,\ a<0.
\]
\begin{proof}[Proof of Theorem~\ref{thm:master}] We treat the four sign cases separately.  
\medskip \\
\noindent\textbf{Case (i):} {$b<0$}.
By Lemma \ref{lem:infinity}, $H^{1}_{\varphi}(\mathbb{R},E[\varphi])$ is trivial. Since $E/\mathbb{Q}$ has good reduction at $q$, then $H^{1}_{\varphi}(\mathbb{Q}_{q},E[\varphi])=H^{1}_{\textup{ur}}(\mathbb{Q}_{q},E[\varphi]$. It follows from \hbox{Lemma \ref{lem:local_conditions_infinity}} and Lemma~\ref{lem:local_condition_cases}(iii), that $\textup{Sel}^{\varphi}(E/\mathbb{Q}) = \textup{Sel}_{\mathcal{S}_{d}}^{\varphi}(E/\mathbb{Q})^{+} \cap \bigcap_{q|d} \textup{ker}(\lambda_q).$

On the other hand, the local conditions for $E^{-d}$ at $q|d$ and at $\infty$ are the trivial groups by Lemma~\ref{loc_conditions_at_q} and Lemma \ref{lem:infinity}. Therefore, by combining \hbox{Corollary \ref{lem:local_conditions_infinity}} and Lemma~\ref{lem:local_condition_cases}(ii):
\begin{align*}
\Sel^{\varphi}(E^{-d}/\mathbb{Q})
&= \Sel^{\varphi}_{\mathcal{S}_d}(E^{-d}/\mathbb{Q})^{+} \cap \bigcap_{q|d} \big{(} \ker(\lambda_{q}) \cap \ker(\gamma_{q}) \big{)}\\
&\stackrel{\textup{Lem. \ref{lem:eq_local_conditions}}}= \Sel^{\varphi}_{\mathcal{S}_d}(E/\mathbb{Q})^{+} \cap \bigcap_{q|d} \big{(} \ker(\lambda_{q}) \cap \ker(\gamma_{q}) \big{)}. 
\end{align*}
By Corollary \ref{cor:gamma_trivial}, $\gamma_q$ is trivial on $ \Sel^{\varphi}_{\mathcal{S}_d}(E/\mathbb{Q})^{+} \cap \bigcap_{q|d}  \ker(\lambda_{q})$, which establishes the required equality.
\medskip \\
\noindent \textbf{Case (ii):} $\Delta_{E}<0$, $b>0$. By combining Lemma~\ref{lem:infinity} with Lemma \ref{lem:local_conditions_infinity}, we deduce
$\Sel^{\varphi}_{\mathcal{S}_{d}}(E/\mathbb{Q})=\Sel^{\varphi}_{\mathcal{S}_{d}^{0}}(E/\mathbb{Q})$. Since $E/\mathbb{Q}$ has good reduction at $q$, then the required description for $\Sel^{\varphi}(E/\mathbb{Q})$ follows by repeated application of Lemma \ref{lem:local_condition_cases}(iii).
For the twist $E^{-d}$, Corollary~\ref{cor:infinity_under_twist} gives 
$\mathcal{L}_{\varphi,\infty}(E^{-d})=H^{1}(\mathbb{R},E[\varphi])$, and Lemma~\ref{loc_conditions_at_q} gives 
$\mathcal{L}_{\varphi,q}(E^{-d})=H^{1}(\mathbb{Q}_{q},E[\varphi])$.  
Hence by combining Lemma \ref{lem:local_condition_cases}(i) and Lemma \ref{lem:local_conditions_infinity}(i), 
$
\Sel^{\varphi}(E^{-d}/\mathbb{Q}) 
\;=\; \Sel^{\varphi}_{\mathcal{S}_{d}}(E^{-d}/\mathbb{Q}) \stackrel{\textup{Lem. \ref{lem:eq_local_conditions}}}= \Sel^{\varphi}_{\mathcal{S}_{d}}(E/\mathbb{Q}).
$
\medskip \\
\noindent \textbf{Case (iii):} $\Delta_{E}>0$, $b>0$ and $a<0$.
By Lemma~\ref{lem:infinity}, $\mathcal{L}_{\varphi,\infty}(E/\mathbb{Q})=H^{1}(\mathbb{R},E[\varphi])$, and so no constraint at $\infty$. Arguing as in case (ii), but instead invoking Lemma \ref{lem:local_condition_cases}(iii) for $q|d$,  we get 
\[ \textup{Sel}^{\varphi}(E/\mathbb{Q})=\textup{Sel}_{\mathcal{S}_{d}}^{\varphi}(E/\mathbb{Q}) \cap \bigcap_{q|d} \textup{ker}(\lambda_q). \]
For $E^{-d}$, $H^{1}(\mathbb{R},E^{-d}[\varphi])$ is trivial by Lemma \ref{lem:local_conditions_infinity}, while $H^{1}(\mathbb{Q}_q,E^{-d}[\varphi])$ is given in \hbox{Lemma \ref{loc_conditions_at_q}(iii)}. Then by combining Lemma \ref{lem:local_conditions_infinity}(ii) and Lemma \ref{lem:local_condition_cases}(iv)--(v), we arrive at
\begin{align*} \textup{Sel}^{\varphi}(E^{-d}/\mathbb{Q}) &= \textup{Sel}_{\mathcal{S}_{d}}^{\varphi}(E^{-d}/\mathbb{Q})^{+} \cap \bigcap_{q \in \mathcal{T}_{d}^{+}} \textup{ker}(\gamma_q)  \cap \bigcap_{q \in \mathcal{T}_{d}^{-}} \textup{ker}(\gamma_q + \lambda_q) \\
 &\stackrel{\textup{Lem. \ref{lem:eq_local_conditions}}}= \textup{Sel}_{\mathcal{S}_{d}}^{\varphi}(E/\mathbb{Q})^{+} \cap \bigcap_{q \in \mathcal{T}_{d}^{+}} \textup{ker}(\gamma_q)  \cap \bigcap_{q \in \mathcal{T}_{d}^{-}} \textup{ker}(\gamma_q + \lambda_q).
 \end{align*}
\medskip \\
\noindent \textbf{Case (iv):} $\Delta_{E}>0$, $b>0$ and $a>0$.
Use the local conditions for $q|d$ and at $\infty$ from Lemma \ref{lem:local_conditions_infinity} and Lemma \ref{loc_conditions_at_q}. Then, argue  as in Case (iii).
\end{proof}

\section{Effect of twisting on Selmer ranks}

In this subsection, we record the variations in the $\varphi$–Selmer ranks of $E$.  
These will be described explicitly in terms of the linear–algebraic framework developed in Section~\ref{subsec_linear_algebra}, together with the explicit description of $\Sel^{\varphi}(E/\Q)$ and $\Sel^{\varphi}(E^{-d}/\Q)$ provided by Theorem~\ref{thm:master}.  
In particular, this formulation allows us to express the change in $\varphi$–Selmer ranks under quadratic twisting purely in terms of matrices over $\mathbb{F}_2$.

\subsection{Rank variation under twisting}
Before stating the main result, we introduce the following block matrices that describe the effect of twisting on Selmer ranks.

\begin{notation}\label{not:twist-matrices}
Assume $b>0$ and $\Delta_{E}>0$, and write $\mathcal{T}_{d}^{\pm}$ as in Notation~\ref{not:T_d}.
Set $r_{1}:=|\mathcal{T}_{d}^{+}|$, $r_{2}:=|\mathcal{T}_{d}^{-}|$, and $r:=r_{1}+r_{2}=\omega(d)$.  
All matrices appearing below are over~$\mathbb{F}_{2}$.  
We write $\mathbf{I}_{k}$ for the $k\times k$ identity matrix and 
$\mathbf{1}_{m\times n}$ for the $m\times n$ all–ones matrix.  
For any matrix~$\mathbf{M} \in \textup{Mat}_{m \times n}(\mathbb{F}_{2})$, we denote by $\overline{\mathbf{M}}:=\mathbf{M}+\mathbf{1}_{m\times n}$ its {complement}, obtained by interchanging $0$ and~$1$ entrywise.

\medskip
With rows and columns ordered by $\mathcal{T}_d^{+}\sqcup\mathcal{T}_d^{-}$, define
\[
\mathbf{A}_d
:=
\begin{bmatrix}
\mathbf{A}_{++} & \mathbf{A}_{+-}\\[2pt]

\mathbf{A}_{-+} & \mathbf{A}_{--}+\mathbf{I}_{r_2}
\end{bmatrix}
\in \mathrm{Mat}_{r\times r}(\F_2),
\]
where each block 
\(\mathbf{A}_{i,j}=(a_{p,q})_{p\in\mathcal{T}_d^{i},\,q\in\mathcal{T}_d^{j}}\)
is defined entrywise by
\[
a_{p,q}=
\begin{cases}
0, & p=q\ \text{or}\ \bigl(\tfrac{q}{p}\bigr)=+1,\\[2pt]
1, & \bigl(\tfrac{q}{p}\bigr)=-1.
\end{cases}
\]
Since all $p,q\in\mathcal{T}_d^{+}\cup\mathcal{T}_d^{-}$ satisfy $p\equiv q\equiv7\tmod{8}$,
quadratic reciprocity gives $\bigl(\tfrac{p}{q}\bigr)\bigl(\tfrac{q}{p}\bigr)=-1$ for $p\neq q$.
Consequently, for the off–diagonal blocks one has
$
\mathbf{A}_{-,+}\;=\;\overline{\bigl(\mathbf{A}_{+,-}\bigr)^{\!\textup{Tr}}}.
$
In the sequel we also use the block matrices  
\[
\widetilde{\mathbf{A}}_{d}
:=\bigl[\,\mathbf{1}_{r\times1}\;\big|\;\mathbf{A}_{d}\bigr]
\;\in\;\mathrm{Mat}_{r\times(1+r)}(\F_{2}),
\qquad
\widehat{\mathbf{A}_{d}}
:=\begin{bmatrix}
\mathbf{1}_{1\times r}\\[2pt]
\overline{\mathbf{A}_{d}}
\end{bmatrix}
\;\in\;\mathrm{Mat}_{(1+r)\times r}(\F_{2}).
\]

In addition, we define an auxiliary invariant $\eta_{\varphi}(d)\in\{0,1\},$
which takes the value $1$ if and only if 
$\Sel_{\mathcal{S}_{d}}^{\varphi}(E/\Q)$ contains an element represented by a squarefree integer $\alpha$ satisfying one of the following:
\begin{itemize}
\item $\alpha>0$ and $\alpha$ is divisible by an odd number of primes of $d$;
\item $\alpha<0$ and $\alpha$ is divisible by an even number of primes of $d$.
\end{itemize}
Otherwise, set $\eta_{\varphi}(d)=0$.
\end{notation}

\begin{theorem} \label{thm:master2}
Let $E/\mathbb{Q}$ be given by the affine model
$
E:\ y^{2}=x\bigl(x^{2}+a x+b\bigr),
$
and let $d\in\mathcal{D}_{E/\mathbb{Q}}$.  
Then, one has:
\begin{enumerate}
\item[(i)] If $b<0$ (and hence $\Delta_{E}>0$), then
\[
\textup{rk}_{\varphi}(E^{-d}/\mathbb{Q})
=\textup{rk}_{\varphi}(E/\mathbb{Q}).
\]

\medskip

\item[(ii)] If $\Delta_E<0$ and $b>0$, then
\begin{align*}
\textup{rk}_{\varphi}(E^{-d}/\mathbb{Q}) = \textup{rk}_{\varphi}(E/\mathbb{Q}) + \omega(d).
\end{align*}

\item[(iii)] If $\Delta_{E}>0$, $b>0$, and $a<0$, then
\begin{align*}
\operatorname{rk}_{\varphi}(E^{-d}/\mathbb{Q})
=\operatorname{rk}_{\varphi}(E/\mathbb{Q})
+\omega(d)-
\begin{cases}
1+\operatorname{rank}_{\mathbb{F}_{2}}(\mathbf{A}_{d}), & \text{if } \eta_{\varphi}(d)=1,\\[4pt]
\operatorname{rank}_{\mathbb{F}_{2}}(\widehat{\mathbf{A}_{d}}), & \text{if } \eta_{\varphi}(d)=0.
\end{cases}
\end{align*}

\item[(iv)] If $\Delta_{E}>0$, $b>0, a>0$, then
\begin{align*}
\mathrm{rk}_{\varphi}(E^{-d}/\Q)
=\mathrm{rk}_{\varphi}(E/\Q)
+ \omega(d)+
\begin{cases}
1-\mathrm{rank}_{\F_{2}}\!\bigl(\widetilde{\mathbf{A}}_{d}\bigr),
& \text{if } \eta_{\varphi}(d)=1,\\[4pt]
-\mathrm{rank}_{\F_{2}}\!\bigl(\overline{\mathbf{A}_{d}}\bigr),
& \text{if } \eta_{\varphi}(d)=0.
\end{cases}
\end{align*}
\end{enumerate}
\end{theorem}

\begin{proof}
We consider the four cases of Theorem \ref{thm:master} separately. 
\medskip \\
\textbf{Case (i):} Clear.    \\
\medskip \\
\noindent \textbf{Case (ii):} 
Set $V:=\Sel_{\mathcal{S}_{d}}^{\varphi}(E/\mathbb{Q})\subset\mathbb{Q}^{\times}/\mathbb{Q}^{\times 2}$. 
By Theorem~\ref{thm:master}(ii) and Lemma~\ref{lem:gamma-lambda-matrix}, 
the difference 
$\operatorname{rk}_{\varphi}(E^{-d}/\mathbb{Q})-\operatorname{rk}_{\varphi}(E/\mathbb{Q})$ 
is the rank of $\Lambda_{\mathcal{S}_{d}^{0}}\!\mid_{V}$. Write $\mathcal{S}_{d}^{0}=\{q_{1},\ldots,q_{r}\}$ and let 
$V_{0}:=\langle -q_{1},\ldots,-q_{r}\rangle\subseteq V$ 
(see Lemma~\ref{lem:support}).  
Taking $\Sigma=\mathcal{S}_{d}^{0}\cup\mathcal{P}$ with $\mathcal{P}=\{p\mid 2N_{E}\}$,  
the embedding $\iota:V\to\mathbb{F}_{2}^{\Sigma\cup\{\infty\}}$ of 
Notation~\ref{not:GammaLambda} yields the matrix
\[
\Lambda_{\mathcal{S}_{d}^{0}}
=\begin{bmatrix}
\mathbf{0}_{r\times1} & \mathbf{I}_{r} & \mathbf{0}_{r\times|\mathcal{P}|}
\end{bmatrix},
\]
whose rows are indexed by the primes in $\mathcal{S}_{d}^{0}$ 
and whose columns are ordered as $(\infty\ \mid\ \mathcal{S}_{d}^{0}\ \mid\ \mathcal{P})$.
Let 
$\pi:\mathbb{F}_{2}^{\{\infty\}\cup\Sigma}\to\mathbb{F}_{2}^{\mathcal{S}_{d}^{0}}$ 
be the projection omitting the $\infty$– and $\mathcal{P}$–coordinates.  
Then, we deduce that $\operatorname{rank}(\Lambda_{\mathcal{S}_{d}^{0}}\!\mid_{V})
=\dim_{\mathbb{F}_{2}}\pi(\iota(V)).$
Since $\pi(\iota(V_{0}))=\mathbb{F}_{2}^{\mathcal{S}_{d}^{0}}$, the same holds for $\pi(\iota(V))$, and so $
\operatorname{rank}\bigl(\Lambda_{\mathcal{S}_{d}^{0}}\!\mid_{V}\bigr)
=\omega(d).
$
\medskip \\
\noindent \textbf{Case (iii):}
In what follows, we retain all notation from case (ii).
As in the proof of part (ii),
$\mathrm{rk}_{\varphi}(E/\mathbb{Q})
=\dim_{\mathbb{F}_{2}}V-\omega(d)$. To compute $\mathrm{rk}_{\varphi}(E^{-d}/\Q)$, fix an ordering on the supporting set 
$\Sigma=\{\infty\}\sqcup\mathcal{T}_{d}^{+}\sqcup\mathcal{T}_{d}^{-}\sqcup\mathcal{P}$.
By Lemma~\ref{lem:gamma-lambda-matrix}, 
$\Sel^{\varphi}(E^{-d}/\Q)$ is the kernel of the matrix $\tilde{\mathbf{A}}$, restricted to $\iota(V)$.
With respect to the column ordering 
$(\infty\mid\mathcal{T}_{d}^{+}\mid\mathcal{T}_{d}^{-}\mid\mathcal{P})$, 
$\tilde{\mathbf{A}}$ takes the block form
\[
\widetilde{\mathbf{A}}:=\begin{bmatrix}
S_{\infty}\\[3pt]
\Gamma_{\mathcal{T}_{d}^{+}}\\[3pt]
\Gamma_{\mathcal{T}_{d}^{-}}+\Lambda_{\mathcal{T}_{d}^{-}}
\end{bmatrix}
=\begin{bmatrix}
1 & \mathbf{0}_{1\times r_{1}} & \mathbf{0}_{1\times r_{2}} & \mathbf{0}_{1\times p}\\[4pt]
\mathbf{1}_{r_{1}\times 1} & \mathbf{A}_{+,+} & \mathbf{A}_{+,-} & \mathbf{0}_{r_{1}\times p}\\[4pt]
\mathbf{1}_{r_{2}\times 1} & \mathbf{A}_{-,+} & \mathbf{A}_{-,-}+\mathbf{I}_{r_{2}} & \mathbf{0}_{r_{2}\times p}
\end{bmatrix}\!.
\]
Let 
$j:\F_{2}^{\{\infty\}\cup\Sigma}\twoheadrightarrow
\F_{2}^{\{\infty\}\cup(\mathcal{T}_{d}^{+}\cup\mathcal{T}_{d}^{-})}=\F_{2}^{1+r}$
be the obvious projection. Then
\[
\operatorname{rank}\!\bigl(\widetilde{\mathbf{A}}\!\mid_{V}\bigr)
=\operatorname{rank}\!\bigl(\mathbf{A}\!\mid_{j(V)}\bigr),
\qquad \textup{ where }
\mathbf{A}:=
\begin{bmatrix}
1 & \mathbf{0}_{1\times r}\\[2pt]
\mathbf{1}_{r\times 1} & \mathbf{A}_{d}
\end{bmatrix}.
\]
Let $V_{0}\subseteq V$ be as in Case~(ii). Then $j(V_{0})$ is the subspace of $\F_{2}^{1+r}$ with basis $e_{\infty}+e_{i}$ $(1\le i\le r)$, where $\{e_{\infty}, e_{1}, \dots, e_r\}$ denotes the standard basis for $\mathbb{F}_{2}^{1+r}$. It follows that $j(V_{0})\subseteq j(V)\subseteq \F_{2}^{1+r}$, and exactly one of the following holds:
\begin{enumerate}[leftmargin=*,itemsep=2pt,topsep=2pt]
\item[(i)] $j(V_{0})\subsetneq j(V)=\F_{2}^{1+r}$, hence 
$\operatorname{rank} \ \!\bigl(\mathbf{A}\!\mid_{j(V)}\bigr)=\operatorname{rank} \ \!\bigl(\mathbf{A}\bigr)    =1+\operatorname{rank} \ (\mathbf{A}_{d})$.
\item[(ii)] $j(V_{0})=j(V)\subsetneq \F_{2}^{1+r}$, in which case 
$\mathbf{A}\!\mid_{j(V)}=\widehat{\mathbf{A}_{d}}$ and 
$\operatorname{rank} \ \!\bigl(\mathbf{A}\!\mid_{j(V)}\bigr)=\operatorname{rank}(\widehat{\mathbf{A}_{d}})$.
\end{enumerate}
It follows that
\[
\operatorname{rk}_{\varphi}(E^{-d}/\mathbb{Q})
= \dim(V) -
\begin{cases}
1+\operatorname{rank}(\mathbf{A}_{d}), & \text{if } j(V)\supsetneq j(V_{0}),\\[4pt]
\operatorname{rank}(\widehat{\mathbf{A}_{d}}), & \text{if } j(V)= j(V_{0}).
\end{cases}
\]
Finally, since $j(V)\supsetneq j(V_{0})$ if and only if $\eta_{\varphi}(d)=1$, the result follows. \\
\medskip \\
\noindent \textbf{Case (iv):}
We retain all notation from Case (iii). Then, $\textup{rk}_{\varphi}(E/\mathbb{Q})$ is determined by the kernel of
\[
\begin{bmatrix}
1 & \mathbf{0}_{1\times r_1} & \mathbf{0}_{1\times r_2} & \mathbf{0}_{1\times p}\\[2pt]
\mathbf{0}_{r_1\times 1} & \mathbf{I}_{r_1} & \mathbf{0}_{r_1\times r_2} & \mathbf{0}_{r_1\times p}\\[2pt]
\mathbf{0}_{r_2\times 1} & \mathbf{0}_{r_2\times r_1} & \mathbf{I}_{r_2} & \mathbf{0}_{r_2\times p}
\end{bmatrix}.
\]
Arguing as in part~(iii), the rank of this matrix on $V$ is $\,\omega(d)\,$ if $\eta_{\varphi}(d)=0$ and $\,\omega(d)+1\,$ if $\eta_{\varphi}(d)=1$. Therefore,
$\textup{rk}_{\varphi}(E/\mathbb{Q})=\dim V-\omega(d)-\big(1\text{ if }\eta_{\varphi}(d) = 1,\ \text{else }0\big)$. Adding on, by Theorem~\ref{thm:master}\,(iv) the dimension of $\Sel^{\varphi}(E^{-d}/\Q)$ is given as the kernel of
\[
\mathbf{M}:=\begin{bmatrix}
\mathbf{1}_{r_{1}\times 1} & \mathbf{A}_{+,+} & \mathbf{A}_{+,-} & \mathbf{0}_{r_1 \times p}\\[4pt]
\mathbf{1}_{r_{2}\times 1} & \mathbf{A}_{-,+} & \mathbf{A}_{-,-}+\mathbf{I}_{r_2} & \mathbf{0}_{r_2 \times p}
\end{bmatrix},
\]
restricted on $V$.
Arguing as in part (iii), if $\eta_{\phi}(d)=1$ then 
$
\rk_{\varphi}(E^{-d}/\Q) \;=\; \dim V - \textup{rank}\bigl(\widetilde{\mathbf{A}}_{d}\bigr).$ If $\eta_{\phi}(d)=0$ then $j(V)=j(V_0)$, and the image of $\mathbf{M}$ on $j(V)$ is the column space of $\mathbf{1} + \mathbf{A}_{d} := \overline{\mathbf{A}_{d}}.$
Writing the $i$th column of $\mathbf{A}_{d}$ as $\mathbf{a}_i\in\F_2^{\,r}$ and $\mathbf{1}\in\F_2^{\,r}$ for the all-ones vector, then
$
\mathbf{A}_{d}(e_{\infty}+e_i)\;=\;\mathbf{1}+\mathbf{a}_i \quad (1\le i\le r),
$
so the image on $j(V)$ is the column space of $\mathbf{A}_{d}+\mathbf{1}$. To conclude, \[
\rk_{\varphi}(E^{-d}/\Q)
=\dim_{\F_2}V-
\begin{cases}
\textup{rank}\bigl(\widetilde{\mathbf{A}_{d}}\bigr), & \text{if } \eta_{\varphi}(d)=1,\\[4pt]
\textup{rank}\bigl(\overline{\mathbf{A}_{d}}\bigr), & \text{if } \eta_{\varphi}(d)=0.
\end{cases} 
\]
This completes the proof in all 4 cases. \qedhere
\end{proof}
\begin{remark} \label{remark:-1_in_truncated_Selmer}
We conclude with a brief remark on the invariant $\eta_{\varphi}(d)$ from Theorem~\ref{thm:master2}.  
To find its value, one checks whether there exists a squarefree integer
\[
\alpha=\pm\!\!\prod_{q\in T}q,\qquad T\subseteq\mathcal{S}_{d},
\]
with the sign chosen according to the parity of $|T|$, satisfying the local conditions at all $p\nmid d\infty$. We note that verifying this condition is straightforward in practice. In particular, it is usually enough to test whether $-1\in\Sel_{\mathcal{S}_{d}}^{\varphi}(E/\Q)$, since this already implies $\eta_{\varphi}(d)=1$.  
At odd places of good reduction the local subgroups are unramified, so the condition is automatic for $-1$;  
if $p\equiv1\tmod4$ then $-1\in\Q_{p}^{\times2}$,  
and at primes of split multiplicative reduction of type~$I_{n}$ with $n$ odd one has $\mathcal{L}_{\varphi,p}=H^{1}(\Q_{p},E[\varphi])$ (see the proof of~\cite[Lemma~3.4]{SW12}),  
so again no restriction arises from such places.
\end{remark}

\subsection{Variations in $\hat{\varphi}$–Selmer ranks}
We now turn to the effect of quadratic twisting on $\hat{\varphi}$–Selmer ranks.  
By Theorem~\ref{thm:master2}, the behaviour of the $\varphi$–Selmer ranks under twisting is already completely explicit.  
An analogous description for the dual isogeny follows by arguing as in Corollary~\ref{cor:master-Eprime}, now expressed in terms of the auxiliary invariant $\eta_{\hat{\varphi}}(d)$.  
This formulation, however, is less convenient, since one typically seeks a unified relation governing both $\varphi$– and $\hat{\varphi}$–Selmer ranks.  
To obtain this, we appeal to Cassels’ formula to express both $\delta_{\varphi}(d)$ and $\delta_{\hat{\varphi}}(d)$ solely in terms of $\Sel_{\mathcal{S}_d}^{\varphi}(E/\Q)$.

\begin{theorem}[Cassels, {\cite{Cassels1965}}] \label{thm:Cassels}
For an isogeny $\varphi:E\to E'$ defined over $\mathbb{Q}$, one has
\[
\mathcal{T}(E/E') \;=\;
\prod_{v} \frac{\#\,\mathcal{L}_{\varphi,v}(E)}{2}.
\]
where $
\mathcal{T}(E/E') \;=\; 
\frac{\#\,\Sel^{\varphi}(E/\mathbb{Q})}{\#\,\Sel^{\hat{\varphi}}(E'/\mathbb{Q})}$, and the product runs over all places $v$ of $\mathbb{Q}$.
\end{theorem}

\begin{proof}
This follows by combining Theorem~1.1 with equations (1.22) and (3.4) in \cite{Cassels1965}.
\end{proof}

\begin{lemma}\label{thm:Cassels_twist}
For $d\in\mathcal{D}_{E/\mathbb{Q}}$, one has
\[
\mathcal{T}(E^{-d}/E'^{-d})
\;=\;
\mathcal{T}(E/E')\,
\biggl(
\frac{\#\,\mathcal{L}_{\varphi,q}(E^{-d})}
     {\#\,\mathcal{L}_{\varphi,q}(E)}
\biggr)^{\!\omega(d)}
\,
\frac{\#\,\mathcal{L}_{\varphi,\infty}(E^{-d})}
     {\#\,\mathcal{L}_{\varphi,\infty}(E)},
\]
where $q$ is any prime factor of $d$.
\end{lemma}

\begin{proof}
We compute directly:
\begin{align*}
\mathcal{T}(E^{-d}/E'^{-d})
=\;
&\Biggl(
\prod_{v\nmid d\infty}
\frac{\#\,\mathcal{L}_{\varphi,v}(E^{-d})}{2}
\Biggr)
\Biggl(
\prod_{q\mid d}
\frac{\#\,\mathcal{L}_{\varphi,q}(E^{-d})}{2}
\Biggr)
\Biggl(
\frac{\#\,\mathcal{L}_{\varphi,\infty}(E^{-d})}{2}
\Biggr) \\
\stackrel{\textup{Lem.~\ref{eq_local_conditions}}}{=}\;
&\Biggl(
\prod_{v\nmid d\infty}
\frac{\#\,\mathcal{L}_{\varphi,v}(E)}{2}
\Biggr)
\Biggl(
\prod_{q\mid d}
\frac{\#\,\mathcal{L}_{\varphi,q}(E^{-d})}{2}
\Biggr)
\Biggl(
\frac{\#\,\mathcal{L}_{\varphi,\infty}(E^{-d})}{2}
\Biggr) \\
=\;
&\mathcal{T}(E/E')
\Biggl(
\prod_{q\mid d}
\frac{\#\,\mathcal{L}_{\varphi,q}(E^{-d})}
     {\#\,\mathcal{L}_{\varphi,q}(E)}
\Biggr)
\Biggl(
\frac{\#\,\mathcal{L}_{\varphi,\infty}(E^{-d})}
     {\#\,\mathcal{L}_{\varphi,\infty}(E)}
\Biggr)\\
=&\;
\mathcal{T}(E/E')
\Biggl(
\frac{\#\,\mathcal{L}_{\varphi,q}(E^{-d})}
     {2}
\Biggr)^{\!\omega(d)}
\Biggl(
\frac{\#\,\mathcal{L}_{\varphi,\infty}(E^{-d})}
     {\#\,\mathcal{L}_{\varphi,\infty}(E)}
\Biggr).
\end{align*}
To obtain the last equality, we combine two observations:  
first, by Lemma~\ref{loc_conditions_at_q}, the quantity 
$\#\,\mathcal{L}_{\varphi,q}(E^{-d})$ depends only on the signs of $\Delta_{E}$ and~$b$, 
and not the prime $q$ dividing $d$;  
second, since $E/\Q$ has good reduction at~$q \mid d$, then 
$\mathcal{L}_{\varphi,q}(E)= H^{1}_{\textup{ur}}(\mathbb{Q}_{q},E[\phi])$ has order~$2$ for all \hbox{such $q$.}
\end{proof}
\begin{corollary} \label{cor:cassels}
Let $d\in\mathcal{D}_{E/\mathbb{Q}}$. Then
\[
\operatorname{rk}_{\hat{\varphi}}(E'^{-d}/\mathbb{Q})
-\operatorname{rk}_{\hat{\varphi}}(E'/\mathbb{Q})
=\;
\operatorname{rk}_{{\varphi}}(E^{-d}/\mathbb{Q})
-\operatorname{rk}_{{\varphi}}(E/\mathbb{Q})
+\omega(d)\,\theta{q}
+\theta_{\infty}.
\]
where
\[
\theta_{q}=
\begin{cases}
+1, & b<0,\\
-1, & b>0,\ \Delta_{E}<0,\\
0, & b>0,\ \Delta_{E}>0,\ a<0,\\
0, & b>0,\ \Delta_{E}>0,\ a>0.
\end{cases}
\qquad
\theta_{\infty}=
\begin{cases}
0, & b<0,\\
0, & b>0,\ \Delta_{E}<0,\\
+1, & b>0,\ \Delta_{E}>0,\ a<0,\\
-1, & b>0,\ \Delta_{E}>0,\ a>0.
\end{cases}
\]

\end{corollary}

\begin{proof}
In view of Lemma \ref{thm:Cassels_twist}, we see that
\[\theta_{q}:=\log_{2}\!\frac{2}{\#\,\mathcal{L}_{\varphi,q}(E^{-d})}, \quad
\theta_{\infty}:=\log_{2}\!\frac{\#\,\mathcal{L}_{\varphi,\infty}(E)}{\#\,\mathcal{L}_{\varphi,\infty}(E^{-d})}.\]
The claimed formula follows from Lemma~\ref{loc_conditions_at_q}, Lemma~\ref{lem:infinity}
and Corollary~\ref{cor:infinity_under_twist}.
\end{proof}

\begin{corollary}\label{cor:hatphi-variation-explicit}
Let $d\in\mathcal{D}_{E/\Q}$. 
Then, one has:
\begin{enumerate}
\item[(i)] If $b<0$ (and hence $\Delta_E>0$), then
\[
\operatorname{rk}_{\hat{\varphi}}(E'^{-d}/\Q)
=\operatorname{rk}_{\hat{\varphi}}(E'/\Q)\;+\;\omega(d).
\]

\medskip

\item[(ii)] If $\Delta_E<0$ and $b>0$, then
\[
\operatorname{rk}_{\hat{\varphi}}(E'^{-d}/\Q)
=\operatorname{rk}_{\hat{\varphi}}(E'/\Q).
\]

\medskip

\item[(iii)] If $\Delta_E>0$, $b>0$, and $a<0$, then
\[
\operatorname{rk}_{\hat{\varphi}}(E'^{-d}/\Q)
=\operatorname{rk}_{\hat{\varphi}}(E'/\Q)\;+\;\omega(d)\;-\;
\begin{cases}
\operatorname{rank}_{\F_{2}}\!\bigl(\mathbf{A}_{d}\bigr),
& \text{if } \eta_{\phi}(d)=1,\\[4pt]
-1+\operatorname{rank}_{\F_{2}}\!\bigl(\widehat{\mathbf{A}_{d}}\bigr),
& \text{if } \eta_{\phi}(d)=0.
\end{cases}
\]

\medskip

\item[(iv)] If $\Delta_E>0$, $b>0$, and $a>0$, then
\[
\operatorname{rk}_{\hat{\varphi}}(E'^{-d}/\Q)
=\operatorname{rk}_{\hat{\varphi}}(E'/\Q)\;+\; \omega(d)\; -\; 
\begin{cases}
\operatorname{rank}_{\F_{2}}\!\bigl(\widetilde{\mathbf{A}_d}\bigr),
& \text{if } \eta_{\phi}(d)=1,\\[4pt]
1\;+\;\operatorname{rank}_{\F_{2}}\!\bigl(\overline{\mathbf{A}_{d}}\bigr),
& \text{if } \eta_{\phi}(d)=0.
\end{cases}
\]
\end{enumerate}
\end{corollary}
\begin{proof}
Combine Theorem~\ref{thm:master2} with Corollary \ref{cor:cassels}.
\end{proof}

\section{Compatibility with the parity conjecture}

We now relate show that the explicit formulae for $\delta_{\varphi}(d)$ and $\delta_{\hat{\varphi}}(d)$ are compatible with predictions of the parity conjecture.  
Recall that
\[
\delta_{\varphi}(d)
 = \mathrm{rk}_{\varphi}(E^{-d}/\mathbb{Q})
   - \mathrm{rk}_{\varphi}(E/\mathbb{Q}), \qquad
\delta_{\hat{\varphi}}(d)
 = \mathrm{rk}_{\hat{\varphi}}(E^{\prime -d}/\mathbb{Q})
   - \mathrm{rk}_{\hat{\varphi}}(E'/\mathbb{Q}).
\]

The $p$-parity conjecture asserts that the parity of the $p^{\infty}$-Selmer rank $\mathrm{rk}_{p^{\infty}}(E/\mathbb{Q})$ is determined by the global root number $w(E/\mathbb{Q})$.  
If $K=\mathbb{Q}(\sqrt{-d})$ satisfies the Heegner hypothesis for $E$, then
$
w(E^{-d}/\mathbb{Q})=-w(E/\mathbb{Q}),
$
and therefore
\begin{equation} \label{eqforpinfty}
\mathrm{rk}_{p^{\infty}}(E^{-d}/\mathbb{Q})
-
\mathrm{rk}_{p^{\infty}}(E/\mathbb{Q}) \equiv 1
\mod{2}.
\end{equation}
Hence the $p$-parity conjecture predicts that twisting by such a field reverses the parity of the $p^{\infty}$-Selmer rank. On the other hand, the explicit expressions for $\delta_{\varphi}(d)$ and $\delta_{\hat{\varphi}}(d)$ satisfy the same parity relation at the level of the $\varphi$- and $\hat{\varphi}$-Selmer groups.  
In particular, combining Theorem~\ref{thm:master2} and Corollary~\ref{cor:hatphi-variation-explicit} yields
\begin{equation} \label{eq with deltas}
\delta_{\varphi}(d) + \delta_{\hat{\varphi}}(d) \equiv 1 \mod{2},
\end{equation}
independently of the matrix ranks appearing in these formulae.  
In what follows, we show that for \emph{any} square-free integer $d$ (not necessarily in $\mathcal{D}_{E/\mathbb{Q}}$), we have:
\[
\mathrm{rk}_{2^{\infty}}(E^{-d}/\mathbb{Q}) - \mathrm{rk}_{2^{\infty}}(E/\mathbb{Q}) \equiv \delta_{\varphi}(d) + \delta_{\hat{\varphi}}(d) \mod 2.
\]
As a consequence, the relation~\eqref{eq with deltas} obtained from the explicit formulae for $\delta_{\varphi}(d)$ and $\delta_{\hat{\varphi}}(d)$ is, in this setting, {equivalent} to the assertion made by the $2$-parity conjecture for $E/K$ give in \eqref{eqforpinfty}. We begin by reformulating the $2$-parity conjecture in terms of the $2$-Selmer ranks of $E$ and its quadratic twists.

\begin{lemma} \label{lemma:parity_of_selmer}
Let $E/\mathbb{Q}$ be an elliptic curve with $E(\mathbb{Q})[2] \cong \mathbb{Z}/2\mathbb{Z}$.  
Then
\[
\mathrm{rk}_{2^{\infty}}(E/\mathbb{Q}) \equiv \mathrm{rk}_{2}(E/\mathbb{Q}) + 1 \mod{2}.
\]
\end{lemma}
\begin{proof}
Let $\delta_2(E/\mathbb{Q})$ denote the $\mathbb{Z}_2$-corank of $\Sha(E/\mathbb{Q})[2^{\infty}]$.  
Then
\[
\Sha(E/\mathbb{Q})[2^{\infty}]
 \simeq (\mathbb{Q}_p/\mathbb{Z}_p)^{\delta_2(E/\mathbb{Q})} \oplus T,
\]
where $T$ is a finite abelian $2$-group.  
It follows that
\begin{align*}
\mathrm{rk}_{2}(E/\mathbb{Q})
 &= 1 + \mathrm{rk}(E/\mathbb{Q})
    + \dim_{\mathbb{F}_{p}}\Sha(E)[2] \\
 &= 1 + \mathrm{rk}(E/\mathbb{Q})
    + \delta_2(E/\mathbb{Q})
    + \dim_{\mathbb{F}_{2}}T[2] \\
 &\equiv 1 + \mathrm{rk}(E/\mathbb{Q})
    + \delta_2(E/\mathbb{Q}) \mod{2} \\
 &= 1 + \mathrm{rk}_{2^{\infty}}(E/\mathbb{Q}) \mod{2}.
\end{align*}
The second to last equality follows from \cite{Cas62} since the Cassels pairing restricts to an alternating and non-degenerate pairing on $T[2]$, implying $\dim_{\mathbb{F}_{2}}T[2]\equiv0\tmod{2}$. For the last equality, by definition, one has $\textup{rk}_{2^{\infty}}(E/\mathbb{Q}) = \textup{rk}(E/\mathbb{Q})+\delta_2(E/\mathbb{Q}).$
\end{proof}

\begin{theorem}[\!\!\cite{DD10}] \label{thm:parity}
Let $E/\Q$ be an elliptic curve and let $K=\Q(\sqrt{-d})$ be an imaginary quadratic field in which all primes of bad reduction for $E/\Q$ split.  
Then
\[
\mathrm{rk}_{2}(E^{-d}/\Q)
\equiv
\mathrm{rk}_{2}(E/\Q) + 1
\tmod{2}.
\]
\end{theorem}

\begin{proof}
Since $E$ is defined over $\mathbb{Q}$, the only non-trivial local contribution to its global root number over~$K$ arises from the unique archimedean place of~$K$, hence $w(E/K)=-1$.  
By \cite[Thm.~1.4]{DD10}, it follows that $\mathrm{rk}_{2^{\infty}}(E/K)\equiv1\tmod{2}$, where $\mathrm{rk}_{2^{\infty}}$ denotes the $\mathbb{Z}_p$-corank of the $p^{\infty}$-Selmer group.  
Let $W=\mathrm{Res}_{K/\mathbb{Q}}E_{K}$ denote the Weil restriction of scalars to~$\mathbb{Q}$.  
By \cite[Thm.~4.5]{MRS07}, $W$ is isogenous over~$\mathbb{Q}$ to $E\times E^{-d}$, and therefore
$
\mathrm{rk}_{2^{\infty}}(W/\mathbb{Q})
 = \mathrm{rk}_{2^{\infty}}(E/\mathbb{Q})
   + \mathrm{rk}_{2^{\infty}}(E^{-d}/\mathbb{Q}).
$
Combining \cite[p.~178(a)]{Mil86} with Shapiro’s lemma yields
$\mathrm{rk}_{2^{\infty}}(W/\mathbb{Q})
 = \mathrm{rk}_{2^{\infty}}(E/K)$, hence
\begin{equation}\label{eq:p_infty}
\mathrm{rk}_{2^{\infty}}(E^{-d}/\mathbb{Q})
 \equiv \mathrm{rk}_{2^{\infty}}(E/\mathbb{Q}) + 1 \mod{2}.
\end{equation}
The stated parity relation follows from Lemma \ref{lemma:parity_of_selmer}. \qedhere
\end{proof}

\begin{proposition} \label{prop:delta_phi}
Let $d$ be a squarefree integer. Then
\[
\delta_{\varphi}(d)+\delta_{\hat{\varphi}}(d)\equiv
\mathrm{rk}_{2}\!\left(E^{-d}/\mathbb{Q}\right)-\mathrm{rk}_{2}\!\left(E/\mathbb{Q}\right)\tmod{2}.
\]
\end{proposition}

\begin{proof}Write $\varphi_{\mathbb{Q}}$ for the induced homomorphism
$
\varphi_{\mathbb{Q}} : E(\mathbb{Q}) \longrightarrow E'(\mathbb{Q}).
$
By \cite[p.~584]{DD10}, one has
\begin{equation} \label{e1}
\frac{\#\Sha(E/\mathbb{Q})[\varphi]}{\#\Sha(E'/\mathbb{Q})[\hat{\varphi}]}
\cdot
\frac{\#\operatorname{coker}\varphi_{\mathbb{Q}}}{\#\operatorname{coker}\hat{\varphi}_{\mathbb{Q}}}
\equiv
\frac{Q(\varphi)}{Q(\hat{\varphi})}
\quad \bmod \mathbb{Q}^{\times 2},
\end{equation}
where $Q(\cdot)$ is defined in \cite[Definition~4.1]{DD10}.  
Using \cite[Lemma~4.2(1),(3)]{DD10}, we obtain
\begin{equation} \label{e2}
\frac{Q(\varphi)}{Q(\hat{\varphi})}
\equiv
Q(\varphi)Q(\hat{\varphi})
= Q(\hat{\varphi}\!\circ\!\varphi)
= Q([2]_E)
= 2^{\mathrm{rk}_{2^{\infty}}(E/\mathbb{Q})} \mod \mathbb{Q}^{\times 2}.
\end{equation}
From the exact sequence
\[
0 \to
E(\mathbb{Q}) / \hat{\varphi}\bigl(E'(\mathbb{Q})\bigr)
\to
\Sel^{\varphi}(E/\mathbb{Q})
\to
\Sha(E/\mathbb{Q})[\varphi]
\to 0,
\]
and its analogue for $\hat{\varphi}$, we deduce
\begin{align}
\#\operatorname{coker}(\varphi_{\mathbb{Q}})
&\equiv
2^{\mathrm{rk}_{\varphi}(E/\mathbb{Q})}
\cdot
\#\Sha(E/\mathbb{Q})[\varphi]
\quad \bmod \mathbb{Q}^{\times 2}, \label{eq:cokerphi}\\
\#\operatorname{coker}(\hat{\varphi}_{\mathbb{Q}})
&\equiv
2^{\mathrm{rk}_{\hat{\varphi}}(E'/\mathbb{Q})}
\cdot
\#\Sha(E'/\mathbb{Q})[\hat{\varphi}]
\quad \bmod \mathbb{Q}^{\times 2}. \label{eq:cokerphihat}
\end{align}
Combining \eqref{e1}--\eqref{eq:cokerphihat} yields
\[
\mathrm{rk}_{2^{\infty}}(E/\mathbb{Q})
\equiv
\mathrm{rk}_{\varphi}(E/\mathbb{Q})
-\mathrm{rk}_{\hat{\varphi}}(E'/\mathbb{Q})
\tmod{2}.
\]
By repeating the argument for $E^{-d}$, we deduce that
\[
\mathrm{rk}_{2^{\infty}}(E^{-d}/\mathbb{Q})
\equiv
\mathrm{rk}_{\varphi}(E^{-d}/\mathbb{Q})
-\mathrm{rk}_{\hat{\varphi}}(E^{\prime -d}/\mathbb{Q})
\tmod{2}.
\]
Subtracting the two congruences, and invoking Lemma \ref{lemma:parity_of_selmer} gives the required parity relation. \qedhere
\end{proof}

An immediate corollary of this parity relation is a symmetric identity between the $\varphi$– and $\hat{\varphi}$–kernels of the Tate--Shafarevich groups of $E$ and its quadratic twists~$E^{-d}$.

\begin{corollary}\label{cor:symmetric-sha-ratio}
Let $d$ be squarefree, and put $K=\mathbb{Q}(\sqrt{-d})$.  
If $\Sha(E/K)[2^{\infty}]$ is finite, then
\begin{equation}\label{eq:symmetric-sha}
\frac{\#\Sha(E/\mathbb{Q})[\varphi]}{\#\Sha(E'/\mathbb{Q})[\hat{\varphi}]}
\;\equiv\;
\frac{\#\Sha(E^{-d}/\mathbb{Q})[\varphi]}{\#\Sha(E^{\prime -d}/\mathbb{Q})[\hat{\varphi}]}
\quad \bmod \mathbb{Q}^{\times 2}.
\end{equation}
If, on the other hand, there exists a squarefree $d$ for which~\eqref{eq:symmetric-sha} fails,  
then at least one of $\Sha(E/\mathbb{Q})[2^{\infty}]$ or $\Sha(E^{-d}/\mathbb{Q})[2^{\infty}]$ is infinite.
\end{corollary}
\begin{proof}
Let $\Phi:\Sha(E/\mathbb{Q})[2]\to\Sha(E'/\mathbb{Q})[\hat{\varphi}]$ and
$\Phi_{d}:\Sha(E^{-d}/\mathbb{Q})[2]\to\Sha(E^{\prime -d}/\mathbb{Q})[\hat{\varphi}]$
be the induced maps on Tate--Shafarevich groups.
From the exact sequence
\[
0 \longrightarrow \Sha(E/\mathbb{Q})[\varphi]
\longrightarrow \Sha(E/\mathbb{Q})[2]
\xrightarrow{\ \Phi\ }
\Sha(E'/\mathbb{Q})[\hat{\varphi}]
\longrightarrow \operatorname{coker}(\Phi)
\longrightarrow 0,
\]
one obtains:
\begin{equation}
2^{\dim_{\mathbb{F}_{2}}\Sha(E/\mathbb{Q})[2]}\cdot \#\,\operatorname{coker}(\Phi)
\;\equiv\;
\frac{\#\,\Sha(E/\mathbb{Q})[\varphi]}{\#\,\Sha(E'/\mathbb{Q})[\hat{\varphi}]}
\quad \bmod \mathbb{Q}^{\times 2}.
\end{equation}
Arguing as in Theorem~\ref{thm:parity}, we may rewrite this as
\begin{equation}\label{eq:a1}
2^{\delta_{2}(E/\mathbb{Q})}\cdot \#\,\operatorname{coker}(\Phi)
\;\equiv\;
\frac{\#\,\Sha(E/\mathbb{Q})[\varphi]}{\#\,\Sha(E'/\mathbb{Q})[\hat{\varphi}]}
\quad \bmod \mathbb{Q}^{\times 2}.
\end{equation}
Repeating the computation for the twist and combining yields
\begin{equation}\label{eq:a2}
2^{\delta_{2}(E^{-d}/\mathbb{Q})-\delta_{2}(E/\mathbb{Q})}\cdot \,\frac{\#\operatorname{coker}(\Phi_{d})}{\#\operatorname{coker}(\Phi)} \cdot \frac{\#\,\Sha(E/\mathbb{Q})[\varphi]}{\#\,\Sha(E'/\mathbb{Q})[\hat{\varphi}]}
\;\equiv\;
\frac{\#\,\Sha(E^{-d}/\mathbb{Q})[\varphi]}{\#\,\Sha(E^{\prime -d}/\mathbb{Q})[\hat{\varphi}]}
\quad \bmod \mathbb{Q}^{\times 2}.
\end{equation}

Next, from the exact sequence in \cite[Lemma 6.1]{SchaeferStoll03}, one obtains:
\[
0\to
E'(\mathbb{Q})[\hat{\varphi}]/\varphi(E(\mathbb{Q})[2])
\to
\Sel^{\varphi}(E/\mathbb{Q})
\to
\Sel^{2}(E/\mathbb{Q})
\to
\Sel^{\hat{\varphi}}(E'/\mathbb{Q})
\to
\operatorname{Coker}(\Phi)
\to 0,
\]
we deduce the relation
\[
\mathrm{rk}_{2}(E/\mathbb{Q}) + 1 + \dim_{\mathbb{F}_{2}}\operatorname{coker}(\Phi)
\;=\;
\mathrm{rk}_{\varphi}(E/\mathbb{Q}) + \mathrm{rk}_{\hat{\varphi}}(E'/\mathbb{Q}),
\]
and similarly for $E^{-d}$.
Subtracting the two identities gives
\[
\mathrm{rk}_{2}(E^{-d}/\mathbb{Q}) - \mathrm{rk}_{2}(E/\mathbb{Q})
\;\equiv\;
\delta_{\varphi}(d) + \delta_{\hat{\varphi}}(d)
+ \dim_{\mathbb{F}_{2}}\operatorname{coker}(\Phi_{d})
- \dim_{\mathbb{F}_{2}}\operatorname{coker}(\Phi)
\quad \tmod 2.
\]
By Proposition~\ref{prop:delta_phi}, one has
$\dim_{\mathbb{F}_{2}}\operatorname{coker}(\Phi_{d})
=\dim_{\mathbb{F}_{2}}\operatorname{coker}(\Phi) \tmod 2$. It follows that
\[
\#\,\operatorname{coker}(\Phi_{d})
\;\equiv\;
\#\,\operatorname{coker}(\Phi)
\quad \bmod \mathbb{Q}^{\times 2}.
\]
Combining this with \eqref{eq:a2}, we find
\[
\frac{\#\,\Sha(E/\mathbb{Q})[\varphi]}{\#\,\Sha(E'/\mathbb{Q})[\hat{\varphi}]}
\;\equiv\;
\frac{\#\,\Sha(E^{-d}/\mathbb{Q})[\varphi]}{\#\,\Sha(E^{\prime -d}/\mathbb{Q})[\hat{\varphi}]}
\cdot 2^{\delta_{2}(E^{-d}/\mathbb{Q}) - \delta_{2}(E/\mathbb{Q})}
\quad \bmod \mathbb{Q}^{\times 2}.
\]
If $\Sha(E/K)[2^{\infty}]$ is finite, then both
$\Sha(E/\mathbb{Q})[2^{\infty}]$ and $\Sha(E^{-d}/\mathbb{Q})[2^{\infty}]$ are finite \cite[Lem.~7.1(b)]{Mil86},
in which case $\delta_{2}(E^{-d}/\mathbb{Q}) = \delta_{2}(E/\mathbb{Q})=0$. This gives the first assertion.
Conversely, if this relation doesn't hold modulo rational squares,
then necessarily
$\delta_{2}(E^{-d}/\mathbb{Q})-\delta_{2}(E/\mathbb{Q})\equiv1\tmod2$. As a result, at least one of $\delta_2(E^{-d}/\mathbb{Q})$ or $\delta_2(E/\mathbb{Q})$ is odd, and so at least one of
$\Sha(E/\mathbb{Q})[2^{\infty}]$ or $\Sha(E^{-d}/\mathbb{Q})[2^{\infty}]$
must be infinite. \qedhere
\end{proof}

\end{document}